\documentclass[final,hidelinks,onefignum,onetabnum]{siamart220329}

\usepackage{lipsum}
\usepackage{amsfonts}
\usepackage{graphicx}
\usepackage{epstopdf}
\usepackage{algorithmic}

\usepackage{commath,amsopn}
\usepackage{multirow}
\usepackage{hyperref}
\usepackage{stmaryrd}
\usepackage{xifthen}
\usepackage{booktabs}
\usepackage{graphicx}
\usepackage{subfigure}
\usepackage{epstopdf}
\usepackage{arydshln}
\usepackage{extarrows}
\usepackage[normalem]{ulem} % for \sout
\graphicspath{ {./figures/} }

\usepackage[figuresright]{rotating}

\newcommand{\vecb}{\mathbf{b}}

\newcommand{\vecx}{\mathbf{x}}
\newcommand{\vecr}{\mathbf{r}}
\newcommand{\vecu}{\mathbf{u}}
\newcommand{\vecv}{\mathbf{v}}

\newcommand{\mata}{\mathbf{A}}
\newcommand{\matb}{\mathbf{B}}

\newcommand{\matx}{\mathbf{X}}
\newcommand{\matu}{\mathbf{U}}
\newcommand{\matv}{\mathbf{V}}
\newcommand{\matW}{\mathbf{W}}
\newcommand{\maty}{\mathbf{Y}}

\newcommand{\matA}{\mathbf{A}}
\newcommand{\matB}{\mathbf{B}}

\newcommand{\matD}{\mathbf{D}}

\newcommand{\matG}{\mathbf{G}}
\newcommand{\matI}{\mathbf{I}}

\newcommand{\matK}{\mathbf{K}}

\newcommand{\matR}{\mathbf{R}}
\newcommand{\matM}{\mathbf{M}}
\newcommand{\matN}{\mathbf{N}}

\newcommand{\matQ}{\mathbf{Q}}
\newcommand{\matS}{\mathbf{S}}
\newcommand{\matX}{\mathbf{X}}
\newcommand{\matU}{\mathbf{U}}

\newcommand{\matY}{\mathbf{Y}}

\newcommand{\matOmega}{\boldsymbol{\Omega}}

\newcommand{\tensA}{\mathcal{A}}
\newcommand{\tensB}{\mathcal{B}}
\newcommand{\tensE}{\mathcal{E}}

\newcommand{\tensH}{\mathcal{H}}
\newcommand{\tensI}{\mathcal{I}}

\newcommand{\tensM}{\mathcal{M}}

\newcommand{\tensS}{\mathcal{S}}
\newcommand{\tensT}{\mathcal{T}}
\newcommand{\tensU}{\mathcal{U}}

\newcommand{\tensX}{\mathcal{X}}

\newcommand{\rank}{\mathrm{rank}}

\newcommand{\rmvec}{\mathrm{vec}}
\newcommand{\tr}{\mathrm{tr}}
\newcommand{\St}{\mathrm{St}}

\newcommand{\grad}{\mathrm{grad}}

\newcommand{\subjectto}{\mathrm{s.\,t.}}
\newcommand{\T}{\mathsf{T}}

\newcommand{\frob}{\mathrm{F}}
\newcommand{\proj}{\mathrm{Proj}}

\newcommand{\vecmatW}{{\vec{\matW}}}

\newcommand{\barf}{{\bar{f}}}
\newcommand{\barg}{{\bar{g}}}
\newcommand{\baru}{\bar{\vecu}}
\newcommand{\barv}{\bar{\vecv}}

\newcommand{\bareta}{{\bar{\eta}}}

\newcommand{\bartensH}{{\bar{\tensH}}}

\newcommand{\cond}{\kappa}
\newcommand{\eucmetric}{\mathrm{e}}

\usepackage{tikz}
\usepackage{tikz-3dplot}
\usetikzlibrary{patterns}
\usetikzlibrary{shapes,calc,shapes,arrows}

\makeatletter
\@addtoreset{equation}{section}
\makeatother

\DeclareMathOperator{\dist}{dist}
\DeclareMathOperator{\diag}{diag}

\DeclareMathOperator{\qf}{qf}
\DeclareMathOperator{\ten}{ten}
\DeclareMathOperator{\tangent}{T}

\DeclareMathOperator{\retr}{R}
\DeclareMathOperator{\Hess}{Hess}
\DeclareMathOperator{\sym}{sym}
\DeclareMathOperator{\argmin}{arg\,min}
\DeclareMathOperator{\argmax}{arg\,max}

\newsiamthm{example}{Example}
\newsiamthm{assumption}{Assumption}

\usepackage{lipsum}
\usepackage{amsfonts}
\usepackage{graphicx}
\usepackage{epstopdf}
\usepackage{algorithmic}
\ifpdf
  \DeclareGraphicsExtensions{.eps,.pdf,.png,.jpg}
\else
  \DeclareGraphicsExtensions{.eps}
\fi

\newsiamremark{remark}{Remark}
\newsiamremark{hypothesis}{Hypothesis}
\crefname{hypothesis}{Hypothesis}{Hypotheses}
\newsiamthm{claim}{Claim}

\headers{Preconditioned optimization on product manifolds}{B. Gao, R. Peng, and Y.-x. Yuan}

\title{Optimization on product manifolds under a preconditioned metric\thanks{Submitted to the editors DATE.
\funding{BG was supported by the Young Elite Scientist Sponsorship Program by CAST. YY was supported by the National Natural Science Foundation of China (grant No. 12288201).}}}

\author{Bin Gao\thanks{State Key Laboratory of Scientific and Engineering 
		Computing, Academy of Mathematics and Systems Science, Chinese  Academy of Sciences, China
		({\{gaobin,yyx\}@lsec.cc.ac.cn}).}
    \and Renfeng Peng\thanks{State Key Laboratory of Scientific and Engineering 
    Computing, Academy of Mathematics and Systems Science, Chinese  Academy of Sciences, and University of Chinese Academy of Sciences, China ({pengrenfeng@lsec.cc.ac.cn}).}
	\and Ya-xiang Yuan\footnotemark[2]%\thanks{State Key Laboratory of Scientific and Engineering Computing, Academy of Mathematics and Systems Science, Chinese  Academy of Sciences, China ({yyx@lsec.cc.ac.cn}).} 
}

\usepackage{amsopn}

\ifpdf
\hypersetup{
  pdftitle={An Example Article},
  pdfauthor={D. Doe, P. T. Frank, and J. E. Smith}
}
\fi

\begin{document}

\maketitle

% REQUIRED
\begin{abstract}
  Since optimization on Riemannian manifolds relies on the chosen metric, it is appealing to know that how the performance of a Riemannian optimization method varies with different metrics and how to exquisitely construct a metric such that a method can be accelerated. To this end, we propose a general framework for optimization problems on product manifolds endowed with a preconditioned metric, and we develop Riemannian methods under this metric. Generally, the metric is constructed by an operator that aims to approximate the diagonal blocks of the Riemannian Hessian of the cost function. We propose three specific approaches to design the operator: exact block diagonal preconditioning, left and right preconditioning, and Gauss--Newton type preconditioning. Specifically, we tailor new preconditioned metrics and adapt the proposed Riemannian methods to the canonical correlation analysis and the truncated singular value decomposition problems, which provably accelerate the Riemannian methods. Additionally, we adopt the Gauss--Newton type preconditioning to solve the tensor ring completion problem. Numerical results among these applications verify that a delicate metric does accelerate the Riemannian optimization methods. 
\end{abstract}

% REQUIRED
\begin{keywords}
  Riemannian optimization; preconditioned metric; canonical correlation analysis; singular value decomposition; tensor completion
\end{keywords}

% REQUIRED
\begin{MSCcodes}
  53B21; 65K05; 65F30; 90C30 
\end{MSCcodes}
\section{Introduction}
We consider the optimization problems on product manifolds:
\begin{equation}
    \label{eq: optimization problem}
    \min_{x\in\tensM} f(x),
\end{equation}
where $f$ is a smooth cost function and the search space $\tensM$ is a product manifold, i.e.,
\[\tensM:=\tensM_1\times\tensM_2\times\cdots\times\tensM_K,\]
$\tensM_k$ is a smooth manifold for $k=1,2,\dots,K$ and $K$ is a positive integer. Optimization on product manifolds has a wide variety of applications, including singular value decomposition~\cite{sato2013riemannian}, joint approximate tensor diagonalization problem~\cite{usevich2020approximate}, dimensionality reduction of EEG covariance matrices~\cite{yamamoto2021subspace}, and canonical correlation analysis~\cite{shustin2023riemannian}. In addition, instead of working with full-size matrices or tensors, matrix and tensor decompositions---which decompose a matrix and tensor into smaller blocks---allow us to implement optimization methods on a product manifold in low-rank matrix and tensor completion~\cite{boumal2015low,kasai2016low,dong2022new,cai2022tensor,swijsen2022tensor,gao2024riemannian}.

\paragraph{Related works and motivation}
Riemannian optimization, designing algorithms based on the geometry of a Riemannian manifold $\tensM$, appears to be prosperous in many areas. One can propose Riemannian optimization methods to solve problem~\cref{eq: optimization problem}, e.g., Riemannian gradient descent and Riemannian conjugate gradient methods. We refer to~\cite{absil2009optimization,boumal2023intromanifolds} for a comprehensive overview.

Since different metrics result in different Riemannian gradients and thus distinct Riemannian methods, one is inquisitive about how the performance of a Riemannian method relies on the choice of a metric $g$. Moreover, the condition number of the Riemannian Hessian of the cost function at a local minimizer $x^*$, denoted by $\kappa:=\kappa_g(\Hess_g\!f(x^*))$, affects the local convergence of first-order methods in Riemannian optimization. For instance, in the Euclidean case, i.e., $\tensM=\mathbb{R}^n$, the asymptotic local linear convergence rates of the steepest gradient descent and the conjugate gradient methods for solving the symmetric positive-definite linear systems are $(\kappa-1)/(\kappa+1)$ and $(\sqrt{\kappa}-1)/(\sqrt{\kappa}+1)$ respectively~\cite[Theorems 3.3, Theorem 5.5]{nocedal2006numerical}. In general, the asymptotic local linear convergence rate of a Riemannian gradient descent method was proved to be $1-1/\mathcal{O}(\kappa)$, see, e.g.,~\cite[Chapter 7, Theorem 4.2]{udriste1994convex},~\cite[Theorem 4.5.6]{absil2009optimization}, and~\cite[Theorem 4.20]{boumal2023intromanifolds}. Notice that an appropriate metric $g$ can lead to a smaller condition number. In view of these observations, it is natural to ask:
\begin{center}
  \emph{Can Riemannian optimization methods be accelerated\\ by choosing a metric ``exquisitely''?}
\end{center}
The following example presents a positive answer. 
\begin{example}
  \label{eg: eg1}
  Consider the problem
  \begin{equation*}
    \min\ f(\vecx):=-\vecb^\T\vecx,\quad \subjectto\ \vecx\in\tensM_\matb:=\{\vecx\in\mathbb{R}^n:\ \vecx^\T\matb\vecx=1\},=
  \end{equation*}
  where $\matb\in\mathbb{R}^{n\times n}$ is symmetric positive definite and $\vecb\in\mathbb{R}^{n}$.
  The search space $\tensM_\matb$ is an ellipsoid. The problem has a closed-form solution $\vecx^*=\matb^{-1}\vecb/\|\matb^{-1}\vecb\|_\matb$ with $\|\vecx\|_\matb^2:=\vecx^\T\matb\vecx$. 
  We explore the effect of a family of metrics, 
  \[g_{\lambda,\vecx}(\xi,\eta):=\langle\xi,(\lambda\matI_n+(1-\lambda)\matB)\eta\rangle\quad\text{for tangent vectors } \xi \text{ and }\eta,\]
  to the Riemannian gradient descent (RGD) method and the condition number of $\Hess_{g_{\lambda}}\!f(\vecx^*)$ in ~\cref{fig: Eg 1.1}, where $\lambda\in\mathbb{R}$ such that $\lambda\matI_n+(1-\lambda)\matB$ is positive definite. The left figure depicts the sequences generated by RGD under the Euclidean metric $g_{1,\vecx}(\xi,\eta)=\langle\xi,\eta\rangle$ and the \emph{scaled metric} $g_{0,\vecx}(\xi,\eta)=\langle\xi,\matb\eta\rangle$, and it shows that RGD under the metric $g_{0}$ converges faster than the one under the Euclidean metric. 
  Furthermore, the right figure confirms that the condition number varies with the metrics and $g_{0}$ leads to the smallest condition number. The detailed computation can be found in~\cref{app: eg 1.1}.
  \begin{figure}[htbp]
    \centering
    \subfigure{\includegraphics[width=0.51\textwidth, trim=3.5cm 3cm 3cm 2.3cm, clip]{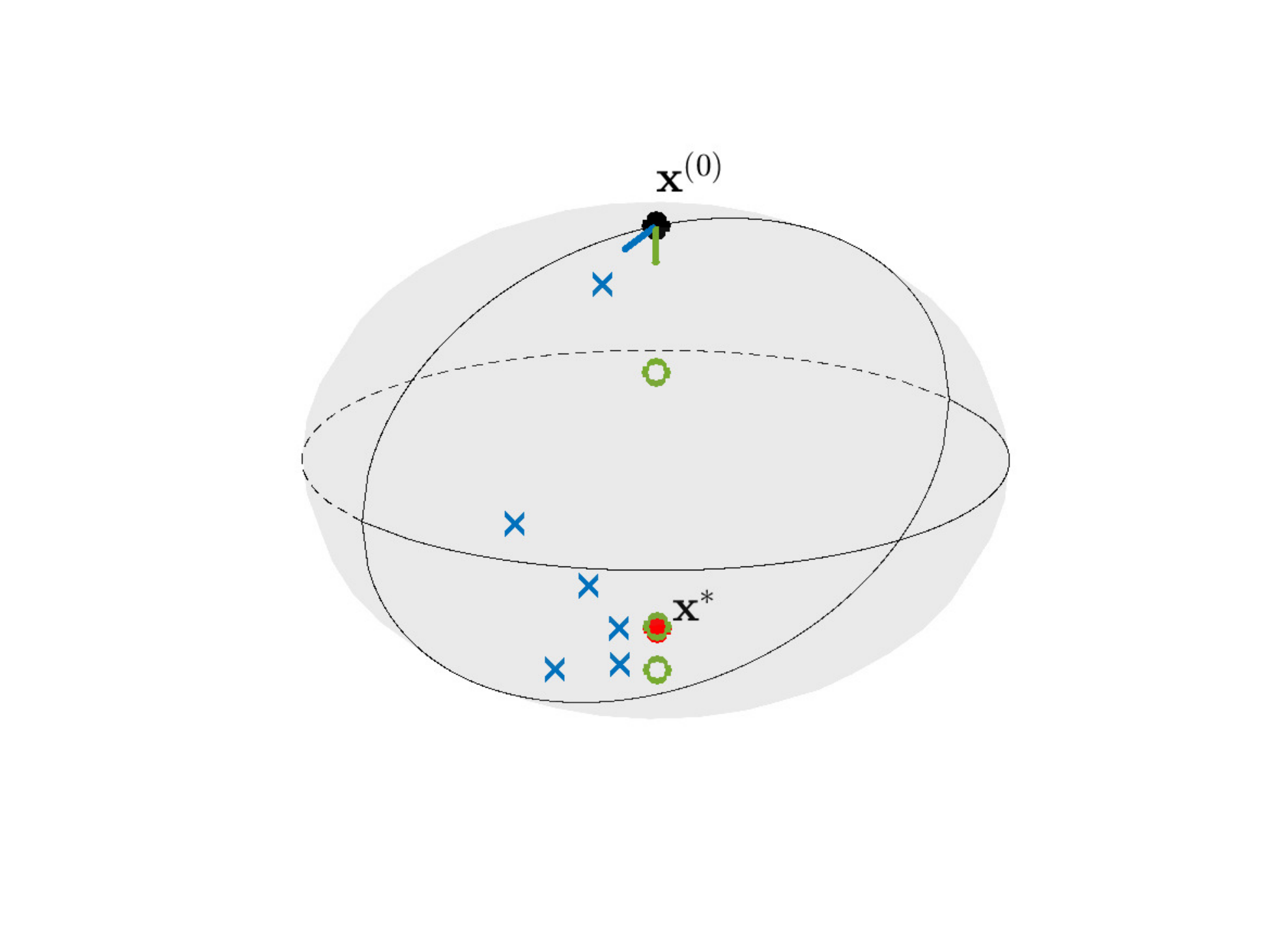}}\quad
    \subfigure{\includegraphics[width=0.44\textwidth]{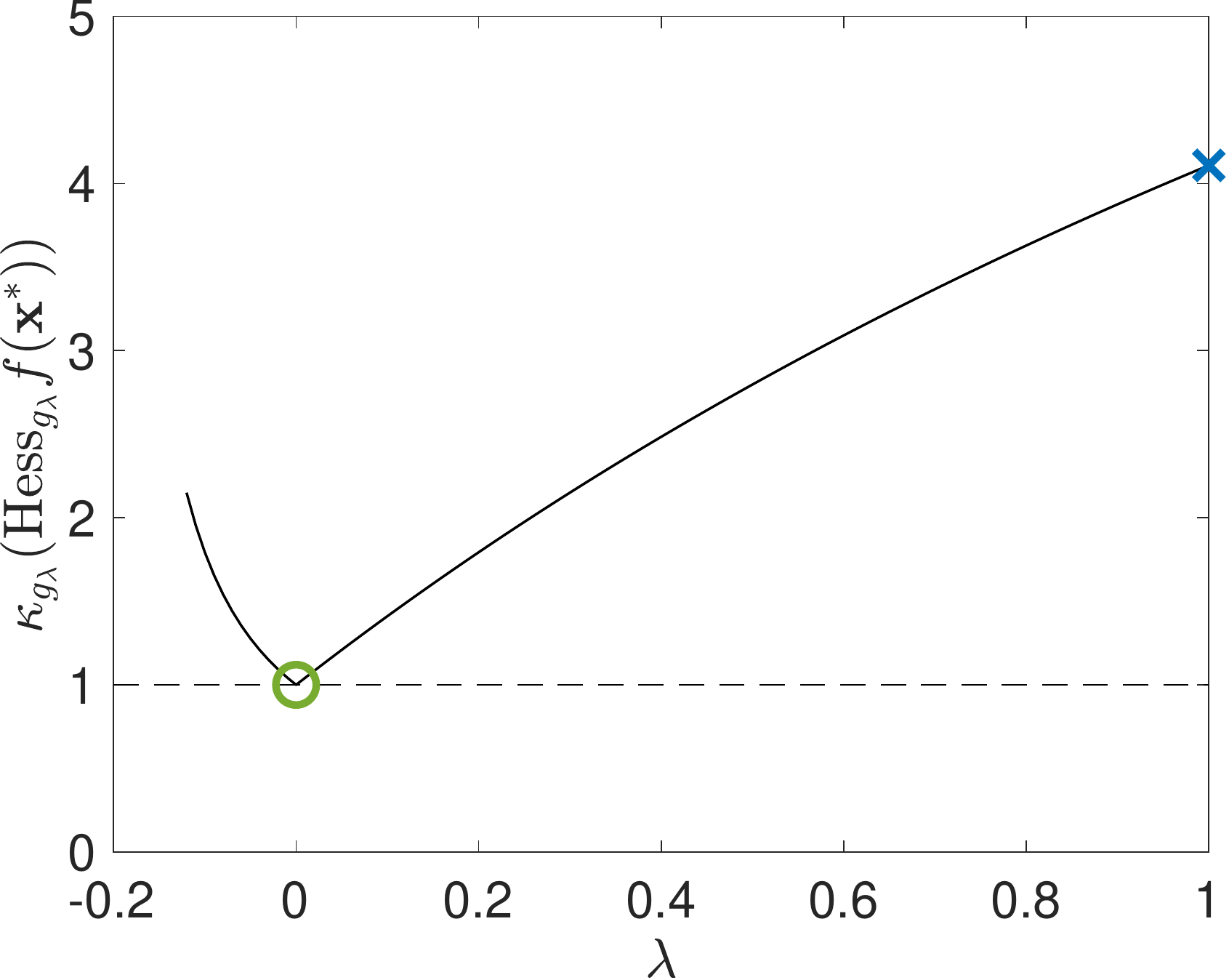}}
    \caption{Left: sequences generated by the Riemannian gradient descent method under two metrics for $\matb=\diag(2^2,3^2,1)$ and $\vecb=(1,1,1)$. Right: the condition number of $\Hess_{g_{\lambda}}\!f(\vecx^*)$ for $\lambda\in(-1/8,1]$. Blue marker: the Euclidean metric; green marker: the scaled metric.}
    \label{fig: Eg 1.1} 
  \end{figure}
\end{example}

Developing an appropriate metric to enhance the performance of Riemannian optimization methods was discussed in the existing works. For instance, \emph{Riemannian preconditioning} was proposed by Mishra and Sepulchre~\cite{mishra2016riemannian} for solving equality-constrained optimization problems where the feasible set enjoys a manifold structure. The non-Euclidean metrics were derived from the Euclidean Hessian of the Lagrangian function, while the explicit construction of the Hessian can be expensive in practice. As a remedy, the \emph{block-diagonal approximation} was considered to construct metrics in the matrix and tensor completion problems~\cite{mishra2012riemannian,kasai2016low,dong2022new,cai2022tensor,gao2024riemannian}. Specifically, in view of the block structure in tensor decompositions, the metric was developed by taking advantage of the diagonal blocks of the Hessian of the cost function, and the Riemannian optimization methods under those metrics were proved to be efficient. 
More recently, Shustin and Avron~\cite{shustin2021faster,shustin2023riemannian} proposed a preconditioned metric for generalized Stiefel manifolds by exploiting the Riemannian Hessian of the cost function at the local minimizer. 

In addition, there are other approaches that incorporate preconditioning techniques in Riemannian optimization. Boumal and Absil~\cite{boumal2015low} developed a preconditioner to approximate the Riemannian Hessian in matrix completion. Kressner et al.~\cite{kressner2016preconditioned} proposed preconditioned Richardson iteration and approximate Newton method to solve tensor equations by constructing a Laplacian-like operator. More recently, Tong et al.~\cite{tong2021accelerating} introduced the scaled gradient descent method for low-rank matrix estimation. Bian et al.~\cite{bian2023preconditioned} presented a preconditioned Riemannian gradient descent algorithm for low-rank matrix recovery. Hamed and Hosseini~\cite{hamed2024riemannian} proposed a Riemannian coordinate descent method under a new Riemannian metric for low multilinear rank approximation.

\paragraph{Contributions}
In this paper, we propose a general framework to construct a preconditioned metric on the product manifold $\tensM=\tensM_1\times\tensM_2\times\cdots\times\tensM_K$, which improves the performance of Riemannian optimization methods. Specifically, we consider a metric by designing a self-adjoint and positive-definite linear operator
$\bartensH$ on the tangent bundle $\tangent\!\tensE$ such that 
\[g_x(\xi,\eta):=\langle\xi,\bartensH(x)[\eta]\rangle\approx\langle\xi,\Hess_{\eucmetric}\!f(x)[\eta]\rangle\quad\text{for }\xi,\eta\in\tangent_{x}\!\tensM,\]
where $\tensE:=\tensE_1\times\tensE_2\times\cdots\times\tensE_K$ is the ambient space of $\tensM$ and $\Hess_{\eucmetric}\!f(x)$ refers to the Riemannian Hessian of $f$ at $x\in\tensM$ under the Euclidean metric $\langle\cdot,\cdot\rangle$. Since the operator $\bartensH(x)$ approximates the second-order information, we refer to the metric as a \emph{preconditioned metric}. 

\begin{figure}[htbp]
  \centering
  \begin{tikzpicture}[scale=0.75]
      \def\dist{3.6};
      \coordinate (M1) at (0,0);
      \coordinate (E1) at ($(M1)-(0.3,0.5)$);

      \coordinate (M2) at ($(M1)+(\dist,0)$);
      \coordinate (E2) at ($(M2)-(0.3,0.5)$);

      \coordinate (Mk) at ($(M1)+(2.6*\dist,0)$);
      \coordinate (Ek) at ($(Mk)-(0.3,0.5)$);
  
      %% The manifold M1 and ambient space E1
      % M1
      \draw[-] plot[smooth, tension=1] coordinates {(M1) ($(M1)+(0.4, 0.55)$) ($(M1)+(1, 0.8)$)};
      \draw[-] plot[smooth, tension=1] coordinates {($(M1)+(1, 0.8)$) ($(M1)+(1.5, 0.6)$) ($(M1)+(1.8, 0.2)$)};
      \draw[-] plot[smooth, tension=1] coordinates {(M1) ($(M1)+(0.4, -0.1)$) ($(M1)+(0.8, -0.4)$)};
      \draw[-] plot[smooth, tension=1] coordinates {($(M1)+(0.8, -0.4)$) ($(M1)+(1.2, 0)$) ($(M1)+(1.8, 0.2)$)};
      \node at ($(M1)+(0.3,1)$) {$\mathcal{M}_1$};
      \node at ($(M1)-(0.2,1)$) {$\mathcal{E}_1$};
      % E1
      \draw[->] (E1) -- ($(E1)-(0.6,0.8)$);
      \draw[->] (E1) -- ($(E1)+(2,0)$);
      \draw[->] (E1) -- ($(E1)+(0,1.5)$);

      %% The manifold M2 and ambient space E2
      % M2
      \node (times1) at ($0.5*(M1)+0.5*(1.8, 0)+0.5*(M2)-0.5*(0.3,0)$) {$\times$};
      \draw[-] plot[smooth, tension=1] coordinates {(M2) ($(M2)+(0.4, 0.55)$) ($(M2)+(1, 0.8)$)};
      \draw[-] plot[smooth, tension=1] coordinates {($(M2)+(1, 0.8)$) ($(M2)+(1.5, 0.6)$) ($(M2)+(1.8, 0.2)$)};
      \draw[-] plot[smooth, tension=1] coordinates {(M2) ($(M2)+(0.4, -0.1)$) ($(M2)+(0.8, -0.4)$)};
      \draw[-] plot[smooth, tension=1] coordinates {($(M2)+(0.8, -0.4)$) ($(M2)+(1.2, 0)$) ($(M2)+(1.8, 0.2)$)};
      \node at ($(M2)+(0.3,1)$) {$\mathcal{M}_2$};
      \node at ($(M2)-(0.2,1)$) {$\mathcal{E}_2$};
      % E2
      \draw[->] (E2) -- ($(E2)-(0.6,0.8)$);
      \draw[->] (E2) -- ($(E2)+(2,0)$);
      \draw[->] (E2) -- ($(E2)+(0,1.5)$);

      %% The manifold Mk and ambient space Ek
      % Mk
      \node (times2) at ($0.5*(M2)+0.5*(1.8, 0)+0.5*(M1)+0.5*(2*\dist,0)-0.5*(0.3,0)$) {$\times$};
      \node (timesk) at ($0.5*(1.6*\dist,0)+0.5*(1.8, 0)+0.5*(Mk)-0.5*(0.3,0)$) {$\times$};
      \node at ($0.5*(times2)+0.5*(timesk)$) {$\cdots$};
      \draw[-] plot[smooth, tension=1] coordinates {(Mk) ($(Mk)+(0.4, 0.55)$) ($(Mk)+(1, 0.8)$)};
      \draw[-] plot[smooth, tension=1] coordinates {($(Mk)+(1, 0.8)$) ($(Mk)+(1.5, 0.6)$) ($(Mk)+(1.8, 0.2)$)};
      \draw[-] plot[smooth, tension=1] coordinates {(Mk) ($(Mk)+(0.4, -0.1)$) ($(Mk)+(0.8, -0.4)$)};
      \draw[-] plot[smooth, tension=1] coordinates {($(Mk)+(0.8, -0.4)$) ($(Mk)+(1.2, 0)$) ($(Mk)+(1.8, 0.2)$)};
      \node at ($(Mk)+(0.3,1)$) {$\mathcal{M}_K$};
      \node at ($(Mk)-(0.2,1)$) {$\mathcal{E}_K$};
      % Ek
      \draw[->] (Ek) -- ($(Ek)-(0.6,0.8)$);
      \draw[->] (Ek) -- ($(Ek)+(2,0)$);
      \draw[->] (Ek) -- ($(Ek)+(0,1.5)$);

      \def\height{-2.3};
      \node (times0) at ($2*(times1)-(times2)-(0.3,0)$) {$=$};
      \node (M) at ($2*(times0)-(M1)+(0.2,0)$) {$\tensM$};
      \node (M) at ($2*(times0)-(M1)+(-0.25,\height)$) {$g_x(\xi,\eta)$};

      \node at ($(times0)+(0,\height)$) {$=$};
      \node (g1) at ($0.5*(times0)+0.5*(times1)+(0,\height)$) {$g^1_{x_1}(\xi_1,\eta_1)$};
      \node at ($(times1)+(0,\height)$) {$+$};
      \node (g2) at ($0.5*(times1)+0.5*(times2)+(0,\height)$) {$g^2_{x_2}(\xi_2,\eta_2)$};
      \node at ($(times2)+(0,\height)$) {$+$};
      \node at ($0.5*(times2)+0.5*(timesk)+(0,\height)$) {$\cdots$};
      \node at ($(timesk)+(0,\height)$) {$+$};
      \node (gK) at ($(timesk)+0.5*(times2)-0.5*(times1)+(0,\height)$) {$g^K_{x_K}(\xi_K,\eta_K)$};

      \def\height1{-3};
      \node at ($(times0)+(0,\height1)$) {$=$};
      \node at ($0.5*(times0)+0.5*(times1)+(0,\height1)$) {$\langle\xi_1,\bartensH_1(x)[\eta_1]\rangle$};
      \node at ($(times1)+(0,\height1)$) {$+$};
      \node at ($0.5*(times1)+0.5*(times2)+(0,\height1)$) {$\langle\xi_2,\bartensH_2(x)[\eta_2]\rangle$};
      \node at ($(times2)+(0,\height1)$) {$+$};
      \node at ($0.5*(times2)+0.5*(timesk)+(0,\height1)$) {$\cdots$};
      \node at ($(timesk)+(0,\height1)$) {$+$};
      \node at ($(timesk)+0.5*(times2)-0.5*(times1)+(0.3,\height1)$) {$\langle\xi_K,\bartensH_K(x)[\eta_K]\rangle$};

  \end{tikzpicture}
  \caption{A new metric on the product manifold $\tensM$.}
  \label{fig: developing metrics on product manifolds}
\end{figure}

We propose three types of preconditioning approaches. The first type is exact block diagonal preconditioning. Instead of approximating the full Riemannian Hessian, which can be computationally unfavorable in practice, we benefit from the block structure of $\Hess_{\eucmetric}\!f(x)$ and construct a new metric by exploiting the diagonal blocks; see an illustration in~\cref{fig: developing metrics on product manifolds}. Specifically, given $x=(x_1,x_2,\dots,x_K)\in\tensM$, we can take advantage of the diagonal blocks $H_{kk}(x)$ of $\Hess_{\eucmetric}\!f(x)$ and construct a metric
\[g_{x_k}^k(\xi_k,\eta_k)=\langle\xi_k,H_{kk}(x)[\eta_k]\rangle\quad\text{for }\xi_k,\eta_k\in\tangent_{x_k}\!\tensM_k\]
if $H_{kk}(x)$ is positive definite; see~\cref{subsec: diagonal approximation}. The second type is left and right preconditioning. If at least one of the diagonal blocks is not positive definite, we construct positive-definite operators $\bartensH_k(x):\tangent_{x_k}\!\tensE_k\to\tangent_{x_k}\!\tensE_k$ for $k=1,2,\dots,K$ that aim to approximate the linear terms in $H_{kk}(x)$ and propose a metric 
\[g_{x_k}^k(\xi_k,\eta_k)=\langle\xi_k,\matM_{k,1}(x)\eta_k\matM_{k,2}(x)\rangle\quad\text{for }\xi_k,\eta_k\in\tangent_{x_k}\!\tensM_k;\] see~\cref{subsec: left and right}. The Riemannian metric on $\tensM$ for both diagonal types is defined by the sum of the Riemannian metric on each component, i.e.,
\[g_x(\xi,\eta)=g_{x_1}^1(\xi_1,\eta_1)+g_{x_2}^2(\xi_2,\eta_2)+\cdots+g_{x_K}^K(\xi_K,\eta_K).\]
Thirdly, we propose Gauss--Newton type preconditioning technique for minimization of $\frac12\|F(x)\|_2^2$ by constructing the metric
\[g_x(\xi,\eta)=\langle\xi,((\mathrm{D}F(x))^*\circ\mathrm{D}F(x))[\eta]\rangle\quad\text{for }\xi,\eta\in\tangent_{x}\!\tensM;\] see~\cref{subsec: Gauss--Newton}.

By virtue of the new metric, we propose the Riemannian gradient descent and Riemannian conjugate gradient methods, and the condition number-related convergence results are developed. The preconditioned metric expands the scope of Riemannian preconditioning in~\cite{mishra2016riemannian} since it facilitates flexible choices of the operator $\bartensH(x)$. It is worth noting that exploiting more second-order information can improve the performance of Riemannian methods, but there is a trade-off between the increased cost brought by preconditioned metrics and the efficiency of preconditioned methods. The existing preconditioning methods that can be interpreted by the proposed framework are listed in~\cref{tab: covered existing works}.

\begin{table}[htbp]
  \centering
  \renewcommand\arraystretch{1.5}
  \scriptsize
  \caption{Existing and our works interpreted by preconditioned metrics. MC: matrix completion; TC: tensor completion; CP: canonical polyadic; TT: tensor train; TR: tensor ring; ``$*$'': non-sigular matrices or tensors. }
  \label{tab: covered existing works}
  \resizebox{\textwidth}{!}{
  \begin{tabular}{llll}
      \toprule
      Problem & Methods & Search space $\tensM$ & Category \\
      \midrule
      MC~\cite{mishra2012riemannian}
       & RGD, RCG, RTR & $\mathbb{R}_*^{m\times r}\times\mathbb{R}_*^{n\times r}$ & Exact block diagonal
      \\ % : $(\matL,\matR)$
      % \midrule[0.2pt]
      Matrix sensing~\cite{tong2021accelerating} & ScaledGD
       & $\mathbb{R}_*^{m\times r}\times\mathbb{R}_*^{n\times r}$ & Exact block diagonal
      \\ % : $(\matL,\matR)$
      % \midrule[0.2pt]
      Tucker TC~\cite{mishra2016riemannian} & RCG & 
      $\times_{k=1}^3\St(r_k,n_k)\times\mathbb{R}^{r_1\times r_2\times r_3}$ & Exact block diagonal 
      \\ % : $(\matu_1,\matu_2,\matu_3,\tensG)$
      % \midrule[0.2pt]
      CP TC~\cite{dong2022new} & RGD, RCG & $\times_{k=1}^K\mathbb{R}^{n_k\times r}$ & Exact block diagonal\\ % : $(\matu_1,\matu_2,\dots,\matu_d)$
      % \midrule[0.2pt]
      {TT} TC~\cite{cai2022tensor} & RGD, RCG, RGN & $\times_{k=1}^K\mathbb{R}_{*}^{r_{k-1}\times n_k\times r_k}$ & Exact block diagonal\\ % : $(\tensX_1,\tensX_2,\dots,\tensX_d)$
      % \midrule[0.2pt]
      {TR} TC~\cite{gao2024riemannian} & RGD, RCG & $\times_{k=1}^K\mathbb{R}^{n_k\times r_{k-1}r_k}$ & Exact block diagonal\\ % : $(\matW_1,\matW_2,\dots,\matW_d)$
        % \midrule[0.2pt]
      CCA~\cite{yger2012adaptive,shustin2023riemannian} & RCG & $\St_{\Sigma_{xx}}(m,d_x)\times\St_{\Sigma_{yy}}(m,d_y)$ & Left and right\\ % : $(\matu,\matv)$
      % \midrule
      {CCA} (this work) & RGD, RCG & $\St_{\Sigma_{xx}}(m,d_x)\times\St_{\Sigma_{yy}}(m,d_y)$ & Left and right\\ % : $(\matu,\matv)$
      % \midrule[0.2pt]
      SVD (this work) & RGD, RCG & $\St(p,m)\times\St(p,n)$ & Left and right\\ % : $(\matu,\matv)$
      % \midrule[0.2pt]
      {TR} TC~(this work) & Gauss--Newton & $\times_{k=1}^K\mathbb{R}^{n_k\times r_{k-1}r_k}$ & Gauss--Newton type\\ % : $(\matW_1,\matW_2,\dots,\matW_d)$
      \bottomrule
  \end{tabular}
  }
\end{table}

Furthermore, we construct novel preconditioned metrics and apply RGD and RCG to canonical correlation analysis (CCA) and truncated singular value decomposition (SVD). We compute the condition numbers of the Riemannian Hessian of the cost function at a local minimizer for these problems. We show that the the proposed metrics indeed improve the condition numbers of Hessian and thus are able to accelerate the Riemannian methods. In addition, we propose the Gauss--Newton method for tensor ring completion. Numerical results among three applications validate the effectiveness of the proposed preconditioning framework, and these methods remain a comparable computational cost with the existing Riemannian methods.

\paragraph{Organization}
We introduce the concepts in Riemannian optimization on product manifolds and present the convergence properties in~\cref{sec: optim on product manifolds}. We propose a general framework and three specific approaches for developing a preconditioned metric on product manifolds in~\cref{sec: develop preconditioned metric on product manifolds}. We apply the proposed framework to solve the canonical correlation analysis and truncated singular value decomposition in~\cref{sec: CCA,sec: SVD}. The Gauss--Newton method for tensor ring completion is proposed in~\cref{sec: LRMC and LRTC}. Finally, we draw the conclusion in~\cref{sec: conclusions}.

\section{Optimization on product manifolds}
\label{sec: optim on product manifolds}
In this section, we provide basic tools in Riemannian geometry on product manifolds and develop the Riemannian gradient descent and Riemannian conjugate gradient methods for optimization on product manifolds; see~\cite{absil2009optimization,boumal2023intromanifolds} for an overview. Metric-based and condition number-related convergence properties are developed.

\subsection{Riemannian optimization on product manifolds}
A product manifold~$\tensM$ is defined by the Cartesian product of manifolds, i.e., \[\tensM=\tensM_1\times\tensM_2\times\cdots\times\tensM_K.\] Assume that $\tensM$ is embedded in a Euclidean space $\tensE=\tensE_1\times\tensE_2\times\cdots\times\tensE_K$, which is called the \emph{ambient space}. 
The tangent space of $\tensM$ at $x=(x_1,x_2,\dots,x_K)$ is denoted by $\tangent_x\!\tensM=\tangent_{x_1}\!\tensM_1\times\tangent_{x_2}\!\tensM_2\times\cdots\times\tangent_{x_K}\!\tensM_K$, and a tangent vector is denoted by $\eta=(\eta_1,\eta_2,\dots,\eta_K)$.
Let each manifold $\tensM_k$ be endowed with a \emph{Riemannian metric}~$g^k$. The Riemannian metric on the product manifold $\tensM$ can be defined by
\[g_x^{}(\xi,\eta):=g_{x_1}^1(\xi_1,\eta_1)+g_{x_2}^2(\xi_2,\eta_2)+\cdots+g_{x_K}^K(\xi_K,\eta_K)\]
for $\xi,\eta\in\tangent_x\!\tensM$, which induces a norm $\|\eta\|_x:=\sqrt{g_x(\eta,\eta)}$.
Given a vector $\bareta=(\bareta_1,\bareta_2,\dots,\bareta_K)\in\tangent_x\!\tensE\simeq\tensE$, the orthogonal projection operator onto $\tangent_x\!\tensM$ with respect to the metric $g$ is $\Pi_{g,x}(\bareta):=(\Pi_{g^1,x_1}(\bareta_1),\Pi_{g^2,x_2}(\bareta_2),\dots,\Pi_{g^K,x_K}(\bareta_K))$,
where each $\Pi_{g^k,x_k}$ refers to orthogonal projection operator with respect to the metric $g^k$ onto $\tangent_{x_k}\!\tensM_k$ for $k=1,2,\dots,K$. Let $\tangent\!\tensM:=\cup_{x\in\tensM}\tangent_x\!\tensM$ be the \emph{tangent bundle}. A smooth mapping $\retr:\tangent\!\tensM\to\tensM$ satisfying $\retr_x(0_x)=x$ and $\mathrm{D}\!\retr_x(0_x)=\mathrm{I}_x$ is called a \emph{retraction}, where $0_x\in\tangent_x\!\tensM$ is the zero element, $\mathrm{D}\!\retr_x(0_x)$ is the differential of $\retr_x$ at $0_x$, and $\mathrm{I}_x:\ \tangent_x\!\tensM\to\tangent_x\!\tensM$ is the identity operator on $\tangent_x\!\tensM$. A retraction on a product manifold $\tensM$ is defined by $\retr_x(\eta):=(\retr_{x_1}^1(\eta_1),\retr_{x_2}^2(\eta_2),\dots,\retr_{x_K}^K(\eta_K))$,
where $\retr^k$ is a retraction on $\tensM_k$.
The \emph{vector transport} operator is denoted by $\tensT_{y\gets x}: \tangent_x\!\tensM\to\tangent_y\!\tensM$ for $x,y\in\tensM$. 

Consider a smooth function $f:\tensM\to\mathbb{R}$. The \emph{Riemannian gradient} under the metric $g$ is denoted by $\grad_g f(x)$, which is the unique tangent vector satisfying $g_x(\grad_g f(x),\eta)=\mathrm{D} f(x)[\eta]$ for all $\eta\in\tangent\!_x\tensM$, where $\mathrm{D}f(x)[\eta]$ refers to the differential of $f$ at $x$ along $\eta$. 
The \emph{Riemannian Hessian} operator of $f$ at $x$ with respect to~$g$ is defined by $\Hess_g\!f(x)[\eta]:=\boldsymbol{\nabla}_\eta\grad_g f$, where $\boldsymbol{\nabla}$ refers to \emph{Levi--Civita connection} on~$\tensM$. If $\tensM$ is a Riemannian submanifold of the Euclidean space $\tensE$, it follows from~\cite[Corollary 5.1.6]{boumal2023intromanifolds} that 
\begin{equation}\label{eq: Riemannian Hessian}
  \Hess_{\eucmetric}\!f(x)[\eta]=\Pi_{\eucmetric,x}(\mathrm{D}\bar{G}(x)[\eta]),
\end{equation}
where $\bar{G}$ is a smooth extension of $\grad_{\eucmetric} f(x)$ on a neighborhood of $\tensM$, $\grad_{\eucmetric} f(x)$ and $\Hess_{\eucmetric}\!f(x)$ are the Riemannian gradient and Riemannian Hessian of $f$ under the Euclidean metric.

By assembling the required ingredients, we present the Riemannian gradient descent and Riemannian conjugate gradient methods in~\cref{alg: RGD,alg: RCG}. 
We refer to~\cite{absil2009optimization,sato2022riemannian} for the global convergence of RGD and RCG.

\begin{algorithm}[htbp]
  \caption{Riemannian gradient descent method (RGD)}\label{alg: RGD}
  \begin{algorithmic}[1]
      \REQUIRE Riemannian manifold $(\tensM,g)$, initial guess $x^{(0)}\in\tensM$, $t=0$. 
      \WHILE{the stopping criteria are not satisfied}
          \STATE Compute $\eta^{(t)}=-\grad_g f(x^{(t)})$.
          \STATE Compute a stepsize $s^{(t)}$.
          \STATE Update $x^{(t+1)}=\retr_{x^{(t)}} (s^{(t)}\eta^{(t)})$; $t= t+1$.
      \ENDWHILE
      \ENSURE $x^{(t)}\in\tensM$. 
  \end{algorithmic}
\end{algorithm}
\begin{algorithm}[htbp]
  \caption{Riemannian conjugate gradient method (RCG)}\label{alg: RCG}
  \begin{algorithmic}[1]
      \REQUIRE Riemannian manifold $(\tensM,g)$, initial guess $x^{(0)}\in\tensM$, $t=0$,  $\beta^{(0)}=0$. 
      \WHILE{the stopping criteria are not satisfied}
          \STATE Compute $\eta^{(t)}=-\grad_g f(x^{(t)})+\beta^{(t)}\tensT_{x^{(t)}\gets x^{(t-1)}}\eta^{(t-1)}$, where $\beta^{(t)}$ is a conjugate gradient parameter.
          \STATE Compute a stepsize $s^{(t)}$. 
          \STATE Update $x^{(t+1)}=\retr_{x^{(t)}} (s^{(t)}\eta^{(t)})$; $t= t+1$.
      \ENDWHILE
      \ENSURE $x^{(t)}\in\tensM$. 
  \end{algorithmic}
\end{algorithm}

Observe that the Riemannian gradients in RGD and RCG depend on the chosen metric $g$. In other words, the Riemannian methods are metric-dependent. Moreover, the computational cost in the updates of RGD and RCG varies with different metrics. Therefore, choosing an appropriate metric is apt to improve the performance of Riemannian methods. 

\begin{definition}[critical points]
  Given a smooth function $f$ defined on a manifold $\tensM$ endowed with a metric $g$, a point $x^*\in\tensM$ is called a critical point of $f$ if~$\grad_g f(x^*)=0$. 
\end{definition}

Note that the definition of Riemannian gradient relies on the metric $g$, whereas the set of critical points of $f$ is invariant to the choice of metrics; see the following proposition.
\begin{proposition}\label{prop: diff metric}
  Given a smooth function $f$ defined on a manifold $\tensM$. Consider two Riemannian manifolds $(\tensM,g)$ and $(\tensM,\tilde{g})$, 
  it holds that
  \[g_{x}^{}(\grad_{g} f(x),\grad_{\tilde{g}} f(x))\geq 0\quad\text{and}\quad \tilde{g}_{x}^{}(\grad_{g} f(x),\grad_{\tilde{g}} f(x))\geq 0\]
  for $x\in\tensM$. The equality holds if and only if $x$ is a critical point. Moreover, $\grad_{g} f(x)=0$ if and only if $\grad_{\tilde{g}} f(x)=0$.
\end{proposition}
\begin{proof}
    See~\cref{app: proof of diff metric}.
\end{proof}

The second-order critical point of $f$ is defined as follows. 
\begin{definition}[second-order critical points]
  Given a smooth function $f$ defined on a manifold $\tensM$ endowed with a metric $g$, a critical point $x^*\in\tensM$ of $f$ is called a second-order critical point of $f$ if $\Hess_g\!f(x^*)$ is positive semidefinite. Furthermore, if $\Hess_g\!f(x^*)$ is positive definite, then $x^*$ is a local minimizer for~\cref{eq: optimization problem}.
\end{definition}

Note that the set of second-order critical points is also invariant in terms of metrics; see~\cite[Proposition 6.3]{boumal2023intromanifolds}. Specifically, if $x^*$ is a second-order critical point of~$f$, it holds that $\Hess_g\!f(x^*)$ is positive semidefinite if and only if  $\Hess_{\tilde{g}}\!f(x^*)$ is positive semidefinite for different metrics $g$ and $\tilde{g}$.

\subsection{Local convergence properties}
\label{subsec: local convergence}
We present the local convergence properties of the Riemannian gradient descent method in terms of condition numbers. Specifically, the Armijo backtracking line search is applied to computing the stepsize in~\cref{alg: RGD}.
\begin{definition}[Armijo backtracking line search]
  Given a smooth function $f$ defined on a manifold $\tensM$ endowed with a metric $g$,  a point $x\in\tensM$, a vector $\eta\in\tangent_x\!\tensM$, an initial stepsize $s_0>0$, and constants $\rho,a\in(0,1)$. The Armijo backtracking line search aims to find the smallest non-negative integer $\ell$, such that for $s=\rho^\ell s_0$, the condition
  \begin{equation}
    \label{eq: Armijo backtracking line search}
    f(x)-f(\retr_x(s\eta))\geq -sag_x(\grad_g f(x),\eta)
  \end{equation} 
  holds.
\end{definition}

In Riemannian optimization, the condition number of the Riemannian Hessian at the local minimizer $x^*$ is crucial to the local converge rate of Riemannian methods; see, e.g.,~\cite[Theorem 4.5.6]{absil2009optimization} and~\cite[Theorem 4.20]{boumal2023intromanifolds}. The condition number of the Riemannian Hessian $\Hess_g\!f(x^*)$ is defined by 
\begin{equation}
    \label{eq: condition number of Hessian}
    \cond_g(\Hess_g\!f(x^*)):=\frac{\lambda_{\max}(\Hess_g\!f(x^*))}{\lambda_{\min}(\Hess_g\!f(x^*))}=\frac{\sup_{\eta\in\tangent_{x^*}\!\tensM}q_{x^*}(\eta)}{\inf_{\eta\in\tangent_{x^*}\!\tensM}q_{x^*}(\eta)},
\end{equation}
where $\lambda_{\min}(\Hess_g\!f(x^*))$ and $\lambda_{\max}(\Hess_g\!f(x^*))$ denote the smallest and largest eigenvalue of $\Hess_g\!f(x^*)$, and
\begin{equation}
    \label{eq: Rayleigh quotient}
    q_{x^*}(\eta):=\frac{g_{x^*}(\eta,\Hess_g\!f(x^*)[\eta])}{g_{x^*}(\eta,\eta)}
\end{equation}
refers to the Rayleigh quotient, which depends on the metric $g$. Then, the local convergence rate of RGD with Armijo backtracking line search~\cref{eq: Armijo backtracking line search} for optimization on product manifolds can be proved by following~\cite[Theorem 4.5.6]{absil2009optimization}. 
\begin{theorem}\label{thm: local convergence}
  Let $\{x^{(t)}\}_{t=0}^{\infty}$ be an infinite sequence generated by~\cref{alg: RGD} with backtracking line search~\cref{eq: Armijo backtracking line search} converging to a local minimizer $x^*$. There exists $T>0$, such that for all $t>T$, it holds that
  \[\frac{f(x^{(t)})-f(x^*)}{f(x^{(t-1)})-f(x^*)}\leq 1-\min\{2as_0\lambda_{\min}(\Hess_g\!f(x^*)),\frac{4a(1-a)\rho}{\kappa_g(\Hess_g\!f(x^*))}\}.\]
\end{theorem}

It is worth noting that different metrics can lead to different $\lambda_{\min}(\Hess_g\!f(x^*))$ and $\kappa_g(\Hess_g\!f(x^*))$ by~\cref{eq: condition number of Hessian}, which affect the local convergence rate. More precisely, a lower condition number indicates a faster convergence in RGD.

\section{Developing preconditioned metric on product manifolds}
\label{sec: develop preconditioned metric on product manifolds}
We first propose a general framework for developing a preconditioned metric on a product manifold $\tensM$ by constructing an operator $\bartensH(x)$ that aims to approximate the diagonal blocks of Riemannian Hessian. Next, we develop three specific approaches to construct the operator $\bartensH(x)$.

Generally, we first propose to endow $\tensE$ with a metric $\barg$ by designing a self-adjoint and positive-definite linear operator $\bartensH$ on the tangent bundle $\tangent\!\tensE$ such that 
\begin{equation}
  \label{eq: original idea of Riemannian preconditioning}
  \barg_x(\xi,\eta)=\langle\xi,\bartensH(x)[\eta]\rangle\approx\langle\xi,\Hess_{\eucmetric}\!f(x)[\eta]\rangle\quad\text{for }\xi,\eta\in\tangent_{x}\!\tensM.
\end{equation}
The rationale behind~\cref{eq: original idea of Riemannian preconditioning} is that the metric $\barg$, aiming to approximate the second-order information of $\Hess_{\eucmetric}\!f(x)$ on $\tangent_x\!\tensM$, is able to facilitate the computation of the Riemannian gradient. Then, the Riemannian metric $g$ on $\tensM$ is inherited from $\tensE$ in view of the Riemannian submanifold. Since $\tensM=\tensM_1\times\tensM_2\times\cdots\times\tensM_K$ is a product manifold, it follows from~\cite[Example 5.19]{boumal2023intromanifolds} that the Riemannian Hessian of~$f$ at $x=(x_1,x_2,\dots,x_K)$ along $\eta=(\eta_1,\eta_2,\dots,\eta_K)$ enjoys a block structure, i.e.,  
\begin{equation}\label{eq: block Hessian}
  \begin{aligned}
    \Hess_{\eucmetric}\!f(x)[\eta]=(&{H_{11}(x)}[\eta_1]~+H_{12}(x)[\eta_2]~+\cdots+H_{1K}(x)[\eta_K],\\
    &H_{21}(x)[\eta_1]~+H_{22}(x)[\eta_2]~+\cdots+H_{2K}(x)[\eta_K],\\
    &\qquad\qquad\qquad\qquad\qquad\ \ ~~\vdots\\
    &H_{K1}(x)[\eta_1]+H_{K2}(x)[\eta_2]+\cdots+{H_{KK}(x)}[\eta_K]),    
  \end{aligned}
\end{equation}
where $H_{ij}(x)[\eta_j]:=\Hess_{\eucmetric}\!f(x_1,\dots,x_{i-1},\cdot,x_{i+1},\dots,x_K)(x_i)[\eta_i]$ if $i=j$, $H_{ij}(x)[\eta_j]:=\mathrm{D} G_i(x_1,\dots,x_{j-1},\cdot,x_{j+1},\dots,x_K)(x_j)[\eta_j]$ if $i\neq j$ and $f(x_1,\dots,x_{i-1},\cdot,x_{i+1},\dots,x_K)$ denotes the function that $f$ is restricted on $\tensM_i$. 
The operator $G_i:\tensM\to\tangent_{x_i}\!\tensM_i$, $G_i(x):=\grad_{\eucmetric} f(x_1,\dots,x_{i-1},\cdot,x_{i+1},\dots,x_K)(x_i)$, gives the Riemannian gradient of the above function. 
The restriction of $G_i$ on $\tensM_j$ is $G_i(x_1,\dots,x_{j-1},\cdot,x_{j+1},\dots,x_K):\tensM_j\to\tangent_{x_i}\!\tensM_i$. In the light of the Riemannian submanifold, the Riemannian gradient of $f$ at $x\in\tensM$ with respect to $g$ can be computed by following~\cite[(3.37)]{absil2009optimization}. 
\begin{proposition}\label{prop: Riemannian gradient as Riemannian submanifold}
  Let $(\tensM,g)$ be a Riemannian submanifold of~$(\tensE,\barg)$. Given a function $f:\tensM\to\mathbb{R}$ and its smooth extension $\barf:\tensE\to\mathbb{R}$, the Riemannian gradient of $f$ at $x\in\tensM$ can be computed by
  \[\grad_g f(x)=\Pi_{g,x}(\bartensH(x)^{-1}[\nabla \barf(x)]),\]
  where $\Pi_{g,x}:\tangent_x\!\tensE\simeq\tensE\to\tangent_x\!\tensM$ is the orthogonal projection operator with respect to the metric $g$ onto $\tangent_x\!\tensM$, and $\nabla \barf(x)$ denotes the Euclidean gradient of $\barf$.  
\end{proposition}

In view of~\cref{eq: original idea of Riemannian preconditioning,prop: Riemannian gradient as Riemannian submanifold}, the operator $\bartensH(x)$ has a preconditioning effect. Hence, we refer to the metric $g$ as a preconditioned metric, and $\bartensH$ is the \emph{preconditioner}. The methodology of using a preconditioned metric can be deemed a general framework to accelerate the Riemannian methods. Subsequently, we design specific approaches for constructing the operator $\bartensH(x)$.

\subsection{Exact block diagonal preconditioning}\label{subsec: diagonal approximation}
Instead of acquiring the full Riemannian Hessian $\Hess_{\eucmetric}\!f(x)$, which involves the computation of all blocks $H_{ij}(x)$ in~\cref{eq: block Hessian}, we develop a metric in a more economical manner by using the diagonal blocks $H_{11},H_{22},\dots,H_{KK}$, as a trade-off between the efficiency and the computational cost. 

Recall that a Riemannian metric on the product manifold $\tensM$ is defined by the sum of the metrics on each component, i.e., $g_x(\xi,\eta)=\sum_{k=1}^K g_{x_k}^k(\xi_k,\eta_k)$ for $\xi=(\xi_1,\xi_2,\dots,\xi_K),\eta=(\eta_1,\eta_2,\dots,\eta_K)\in\tangent_{x}\!\tensM$. 
Consider the smooth extension $\bar{H}_{kk}(x):\tangent_{x_k}\!\tensE_k\to\tangent_{x_k}\!\tensE_k$ of the diagonal block $H_{kk}$ for $k=1,2,\dots,K$. If $\bar{H}_{11},\bar{H}_{22},\dots,\bar{H}_{KK}$ are positive definite on the ambient space $\tensE$, one can immediately adopt these blocks to construct the operator $\bartensH$ by
\begin{equation}
    \label{eq: exact}
    \bartensH(x)[\eta]=(\bartensH_{1}(x)[\eta_1],\dots,\bartensH_{K}(x)[\eta_K])=(\bar{H}_{11}(x)[\eta_1],\dots,\bar{H}_{KK}(x)[\eta_K]).
\end{equation}
Hence, the metric
\begin{equation*}
    \begin{aligned}
        g_x(\xi,\eta)&=g_{x_1}^1(\xi_1,\eta_1)+g_{x_2}^2(\xi_2,\eta_2)+\cdots+g_{x_K}^K(\xi_K,\eta_K)\\
        &=\langle\xi_1,\bar{H}_{11}(x)[\eta_1]\rangle+\langle\xi_2,\bar{H}_{22}(x)[\eta_2]\rangle+\cdots+\langle\xi_K,\bar{H}_{KK}(x)[\eta_K]\rangle
    \end{aligned}
\end{equation*}
is a well-defined Riemannian metric on $\tensM$, which leads to \emph{exact block diagonal} preconditioning. Note that the proposed Exact block diagonal preconditioning is not applicable when at least one of the blocks $\bar{H}_{11},\bar{H}_{22},\dots,\bar{H}_{KK}$ is not positive definite. In practice, one can consider a regularization term $\delta_k\tensI_k(x)$ with the identity operator $\tensI_k(x):\tangent_{x_k}\!\tensE_k\to\tangent_{x_k}\!\tensE_k$ and some $\delta_k>0$ to ensure that the operator $\bar{H}_{kk}(x)+\delta_k\tensI_k(x)$ is positive definite.

In contrast with the Riemannian Hessian~\cref{eq: block Hessian}, which contains cross terms among $\eta_1,\eta_2,\dots,\eta_K$, the operator $\bartensH(x)$ in~\cref{eq: exact} enjoys a block structure. Therefore, the Riemannian gradient of a smooth function $f:\tensM\to\mathbb{R}$ at $x\in\tensM$ can be computed on each block in view of~\cref{prop: Riemannian gradient as Riemannian submanifold}.

\begin{proposition}\label{prop: Riemannian gradient as Riemannian submanifold of product space}
  Let $\tensM=\tensM_1\times\tensM_2\times\cdots\times\tensM_K$ be a product manifold endowed with the metric~$g$.
  Given a function $f:\tensM\to\mathbb{R}$ and its smooth extension $\barf:\tensE\to\mathbb{R}$, the Riemannian gradient of $f$ at $x$ is 
  \begin{equation}\label{eq: Riemannian gradient on product space}
    \grad_g f(x)=\big(\Pi_{g^1,x_1}(\bartensH_1(x)^{-1}[\partial_1 \barf(x)]),\dots,\Pi_{g^K,x_K}(\bartensH_K(x)^{-1}[\partial_K \barf(x)])\big),
  \end{equation}
  where $\Pi_{g^k,x_k}$ is the orthogonal projection operator with respect to the metric $g^k$ onto $\tangent_{x_k}\!\tensM_k$ for $k=1,2,\dots,K$ and $\partial_k\barf(x)$ is the partial derivative of $f$ with respect to~$x_k$.  
\end{proposition}

It is worth noting that developing an appropriate metric by exploiting the diagonal blocks is closely related to the \emph{block-Jacobi} preconditioning~\cite{demmel2023nearly} in numerical linear algebra. Specifically, given a symmetric positive definite matrix $\mata\in\mathbb{R}^{n\times n}$, the goal of block-Jacobi preconditioning is to construct an invertible block-diagonal matrix 
  \[\matD=\begin{bmatrix}
      \matD_{11} &  & & \\
      & \matD_{22} &  & \\
      &  & \ddots & \\
      &  &  & \matD_{KK}\\
    \end{bmatrix}\in\mathbb{R}^{n\times n}\ \text{with}\ \matD_{kk}\in\mathbb{R}^{n_k\times n_k},\ k=1,2,\dots,K
  \]
  such that $\kappa_2(\matD\mata\matD^\T):=\lambda_{\max}(\matD\mata\matD^\T)/\lambda_{\min}(\matD\mata\matD^\T)$ is reduced, where $n_1+n_2+\cdots+n_K=n$. Alternatively, consider the minimization problem of a quadratic function $\min_{\vecx\in\mathbb{R}^n}f(\vecx):=\frac12\vecx^\T\mata\vecx$. We can construct a preconditioned metric on the product manifold $\mathbb{R}^n=\mathbb{R}^{n_1}\times\mathbb{R}^{n_2}\times\cdots\times\mathbb{R}^{n_K}$ by
  \[g_\vecx(\xi,\eta):=\sum_{k=1}^K\langle\xi_k,(\matD_{kk}^\T\matD_{kk}^{})^{-1}\eta_k\rangle=\langle\xi,(\matD^\T\matD)^{-1}\eta\rangle.\]
  Given $\vecx\in\mathbb{R}^n$, it follows from~\cref{prop: Riemannian gradient as Riemannian submanifold of product space} and the definition of Riemannian Hessian that $\grad_g f(\vecx)=(\matD^\T\matD)\mata\vecx$ and $\Hess_g\!f(\vecx)=(\matD^\T\matD)\mata$. Therefore, the Rayleigh quotient is given by
  \begin{equation*}
    q_\vecx(\eta)=\frac{g_\vecx(\eta,\Hess_g\!f(\vecx)[\eta])}{g_\vecx(\eta,\eta)}=\frac{\langle\eta,\mata\eta\rangle}{\langle\eta,(\matD^\T\matD)^{-1}\eta\rangle}\xlongequal{\tilde{\eta}:=\matD^{-\T}\eta}\frac{\langle\tilde{\eta},(\matD\mata\matD^\T)\tilde{\eta}\rangle}{\langle\tilde{\eta},\tilde{\eta}\rangle}
  \end{equation*}
  for $\eta\in\tangent_\vecx\!\mathbb{R}^n\simeq\mathbb{R}^n$. It follows from~\cref{eq: condition number of Hessian} that 
  \[\kappa_g(\Hess_g\!f(\vecx))=\frac{\sup_{\eta\in\tangent_{\vecx}\!\tensM}q_{\vecx}(\eta)}{\inf_{\eta\in\tangent_{\vecx}\!\tensM}q_{\vecx}(\eta)}=\frac{\lambda_{\max}(\matD\mata\matD^\T)}{\lambda_{\min}(\matD\mata\matD^\T)}=\kappa_2(\matD\mata\matD^\T).\]
  As a result, the block-Jacobi preconditioning that aims to reduce $\kappa_2(\matD\mata\matD^\T)$ is equivalent to selecting a specific preconditioned metric on $\mathbb{R}^n$ to reduce the condition number of the Riemannian Hessian of $f$, i.e., $\kappa_g(\Hess_g\!f(\vecx))$. Additionally, the preconditioning in matrix and tensor completion can be interpreted as the proposed exact block diagonal preconditioning; see~\cref{sec: LRMC and LRTC} for details.

\subsection{Left and right preconditioning}\label{subsec: left and right}
Generally, in block diagonal preconditioning, $\bar{H}_{11},\bar{H}_{22},\dots,\bar{H}_{KK}$ are not positive definite on the ambient space $\tensE$. Therefore, we seek an appropriate approximation of these terms. 

Specifically, we assume that $\tensE$ is a product space of matrices and ${H}_{kk}(x)[\eta_k]$ contains a linear term $\bar{\matM}_{k,1}(x)\eta_k\bar{\matM}_{k,2}(x)$, where $\bar{\matM}_{k,1}(x)$ and $\bar{\matM}_{k,2}(x)$ are square matrices for fixed $x\in\tensE$. In view of the Riemannian Hessian~\cref{eq: block Hessian}, we approximate the diagonal block ${H}_{kk}(x)$ via $\bar{\matM}_{k,1}(x)\eta_k\bar{\matM}_{k,2}(x)$ by constructing an operator $\bartensH_k(x):\tangent_{x_k}\!\tensE_k\to\tangent_{x_k}\!\tensE_k$ such that 
\[\langle\xi_k,\bartensH_k(x)[\eta_k]\rangle=\langle\xi_k,\matM_{k,1}(x)\eta_k\matM_{k,2}(x)\rangle\approx\langle\xi_k,{H}_{kk}(x)[\eta_k]\rangle\quad\text{for }\xi_k,\eta_k\in\tangent_{x_k}\!\tensM_k, \]
where $\matM_{k,j}(x)=(\sym(\bar{\matM}_{k,j}(x))^2+\delta\matI)^{1/2}$ for $j=1,2$, $\sym(\mata):=(\mata+\mata^\T)/2$ and $\delta>0$ to ensure the positive definiteness.
Subsequently, a \emph{left and right} preconditioned metric on $\tangent_x\!\tensM$ is given by
\begin{align}
    g_x(\xi,\eta)&=g_{x_1}^1(\xi_1,\eta_1)+\cdots+g_{x_K}^K(\xi_K,\eta_K)\nonumber\\
    &=\langle\xi_1,\matM_{1,1}(x)\eta_1\matM_{1,2}(x)\rangle+\cdots+\langle\xi_K,\matM_{K,1}(x)\eta_K\matM_{K,2}(x)\rangle.
    \label{eq: preconditioned metric on product space}
\end{align}
Note that $\matM_{k,j}(x)$ are smooth and positive definite for all $x\in\tensE$ and thus~\cref{eq: preconditioned metric on product space} is a well-defined Riemannian metric. The corresponding Riemannian gradient can be also computed by following~\cref{prop: Riemannian gradient as Riemannian submanifold of product space} since the operator $\bartensH$ is defined on each block.

We adopt the proposed left and right preconditioning to accelerate the Riemannian methods for canonical correlation analysis (CCA); see~\cref{sec: CCA} for details. In practice, we can consider only the left or right preconditioning to save the computational cost. Section~\ref{sec: SVD} presents how we develop right preconditioning for truncated singular value decomposition (SVD). It is worth noting that if the operators $\matM_{k,j}$ in~\cref{eq: preconditioned metric on product space} are not chosen appropriately, Riemannian methods under the metric~\cref{eq: preconditioned metric on product space} can even perform worse than those under the Euclidean metric; see~\cref{subsec: numerical CCA}. Nevertheless, the proposed metrics tailored for CCA and SVD improve the condition number of $\Hess\!f(x)$, thereby accelerating the Riemannian methods indeed; see~\cref{prop: better condition number}.

\subsection{Gauss--Newton type preconditioning}
\label{subsec: Gauss--Newton}
If the cost function in~\cref{eq: optimization problem} satisfies that $f(x):=\frac12\|F(x)\|_\mathrm{F}^2$ for some smooth function $F:\tensM\to\mathbb{R}^n$ with injective $\mathrm{D}F(x)$, one can consider the operator $\bartensH(x)=(\mathrm{D}F(x))^*\circ\mathrm{D}F(x)$ to approximate $\Hess_\eucmetric\!f(x)$, and construct the preconditioned metric as follows,
\begin{equation}
  \label{eq: Gauss--Newton metric}
  g_x(\xi,\eta)=\langle\xi,\bartensH(x)[\eta]\rangle=\langle\xi,((\mathrm{D}F(x))^*\circ\mathrm{D}F(x))[\eta]\rangle\approx\langle\xi,\Hess_\eucmetric\!f(x)[\eta]\rangle,
\end{equation}
where $(\mathrm{D}F(x))^*$ is the adjoint operator of $\mathrm{D}F(x)$.
The \emph{Gauss--Newton type} preconditioning is no longer a block diagonal preconditioning since $\bartensH(x)$ contains cross terms among $\eta_1,\eta_2,\dots,\eta_K$. As a result, the Riemannian gradient directly follows from~\cref{prop: Riemannian gradient as Riemannian submanifold}.

Note that the Riemannian gradient descent method with the metric $g$ is exactly the Riemannian Gauss--Newton method~\cite[\S 8.4.1]{absil2009optimization}, where the search direction $\eta^{(t)}\in\tangent_{x^{(t)}}\!\tensM$ at $x^{(t)}\in\tensM$ is computed by the following Gauss--Newton equation
\[\langle\mathrm{D}F(x^{(t)})[\xi],\mathrm{D}F(x^{(t)})[\eta^{(t)}]\rangle+\langle\mathrm{D}F(x^{(t)})[\xi],F(x^{(t)})\rangle=0\quad\text{for }\xi\in\tangent_{x^{(t)}}\!\tensM,\]
or $((\mathrm{D}F(x^{(t)}))^*\circ\mathrm{D}F(x^{(t)}))[\eta^{(t)}]=-(\mathrm{D}F(x^{(t)}))^*[F(x^{(t)})]$. It follows from the injectivity of $\mathrm{D}F(x^{(t)})$ that
\[\eta^{(t)}=-((\mathrm{D}F(x^{(t)}))^*\circ\mathrm{D}F(x^{(t)}))^{-1}[(\mathrm{D}F(x^{(t)}))^*[F(x^{(t)})]],\]
which is also the solution of the following least-squares problem
\begin{equation}\label{eq: search direction of GN}
  \min_{\eta\in\tangent_{x^{(t)}}\!\tensM}\frac12\langle\mathrm{D}F(x^{(t)})[\eta],\mathrm{D}F(x^{(t)})[\eta]\rangle+\langle\mathrm{D}F(x^{(t)})[\eta],F(x^{(t)})\rangle.
\end{equation}

Since $\langle\mathrm{D}F(x^{(t)})[\eta^{(t)}],F(x^{(t)})\rangle=\mathrm{D}f(x^{(t)})[\eta^{(t)}]=\mathrm{D}\barf(x^{(t)})[\eta^{(t)}]=\langle\nabla\barf(x^{(t)}),\eta^{(t)}\rangle$, where $\barf:\tensE\to\mathbb{R}$ is any smooth extension of $f$,~\cref{eq: search direction of GN} is equivalent to
\begin{equation}\label{eq: search direction of GN new}
  \min_{\eta\in\tangent_{x^{(t)}}\!\tensM}\frac12\langle\bartensH(x^{(t)})[\eta],\eta\rangle+\langle\nabla\barf(x^{(t)}),\eta\rangle.
\end{equation}
It follows from~\cite{mishra2016riemannian} that the solution of~\cref{eq: search direction of GN new} is $\eta^{(t)}=-\grad_g f(x^{(t)})$. In other words, the Riemannian Gauss--Newton method can be interpreted by the Riemannian gradient descent method with the metric $g$. Therefore, we refer to the proposed framework as the Gauss--Newton type preconditioning, which can be adopted to tensor completion; see~\cref{sec: LRMC and LRTC} for details.

\begin{remark}
  Let the operator $\bartensH(x)$ be the Riemannian Hessian $\Hess_\eucmetric\!f(x)$ of $f$ at $x\in\tensM$ under the Euclidean metric. If $\Hess_\eucmetric\!f(x)$ is positive definite, the metric $g_x(\xi,\eta)=\langle\xi,\Hess_\eucmetric\!f(x)[\eta]\rangle$ is referred to \emph{Hessian metric} in~\cite{shima1997geometry}. The negative Riemannian gradient $-\grad_g f(x)$ echoes the Riemannian Newton direction under the Euclidean metric; see~\cite[Proposition 4.1]{absil2009all} and~\cite[Proposition 2.1]{mishra2016riemannian}. Note that the proposed Gauss--Newton type preconditioning is different from the Hessian metric since~\cref{eq: Gauss--Newton metric} adopts partial information of the Riemannian Hessian.
\end{remark}

\section{Application to canonical correlation analysis}\label{sec: CCA} 
In this section, we apply the proposed framework to solve the canonical correlation analysis (CCA) problem. A new left and right preconditioned metric is proposed. Then, we prove that the the proposed metric improves the condition number of the Riemannian Hessian. Numerical experiments verify that the proposed metric accelerates the Riemannian methods. 

Consider two data matrices $\matX\in\mathbb{R}^{n\times d_x}$ and $\matY\in\mathbb{R}^{n\times d_y}$ with $n$ samples and $d_x$, $d_y$ variables respectively. The goal of CCA is to choose $m$ weights $\vecu_1,\dots,\vecu_m\in\mathbb{R}^{d_x}$ and $\vecv_1,\dots,\vecv_m\in\mathbb{R}^{d_y}$ such that the data matrices $\matx\matu$ and $\maty\matv$ have the highest correlation, where $\matu=[\vecu_1,\dots,\vecu_m]$ and $\matv=[\vecv_1,\dots,\vecv_m]$. 
CCA can be written as an optimization problem on the product manifold of two \emph{generalized Stiefel manifolds}, i.e,
\begin{equation}\label{eq: CCA}
  \min_{\matu,\matv}\ f(\matu,\matv):=-\tr(\matu^\T\Sigma_{xy}\matv\matN),\ \subjectto\ (\matu,\matv)\in\tensM=\tensM_1\times\tensM_2,
\end{equation}
where $\Sigma_{xx}:=\matx^\T\matx+\lambda_x\matI_{d_x}$, $\Sigma_{yy}:=\maty^\T\maty+\lambda_y\matI_{d_y}$, 
$\lambda_x,\lambda_y\geq 0$ are regularization parameters, $\Sigma_{xy}:=\matx^\T\maty$, $\tensM_1:=\St_{\Sigma_{xx}}(m,d_x)=\{\matu\in\mathbb{R}^{d_x\times m}:\matu^\T\Sigma_{xx}\matu=\matI_m\}$ and $\tensM_2:=\St_{\Sigma_{yy}}(m,d_y)$ refer to the generalized Stiefel manifolds, and $\matN:=\diag(\mu_1,\mu_2,\dots,\mu_m)$ satisfies $\mu_1>\mu_2>\cdots>\mu_m>0$. The cost function $f$ in~\cref{eq: CCA} is also known as the \emph{von Neumann cost function}~\cite{yon1937some}. The problem~\cref{eq: CCA} has a closed-form solution 
\begin{equation}\label{eq: cca solution}
  (\matu^*,\matv^*)=(\Sigma_{xx}^{-1/2}\bar{\matu},\Sigma_{yy}^{-1/2}\bar{\matv}),
\end{equation}
where $\bar{\matu}:=[\baru_1,\dots,\baru_m]$ and $\bar{\matv}:=[\barv_1,\dots,\barv_m]$ are the $m$ leading left and right singular vectors of the matrix $\Sigma_{xx}^{-1/2}\Sigma_{xy}\Sigma_{yy}^{-1/2}$ respectively. The $m$ largest singular values $\sigma_1\geq\sigma_2\geq\cdots\geq\sigma_m>0$ of $\Sigma_{xx}^{-1/2}\Sigma_{xy}\Sigma_{yy}^{-1/2}$ are referred to as the canonical correlations. 
We intend to propose preconditioned metrics on $\tensM$ and adapt the Riemannian methods to solve~\cref{eq: CCA}.

\subsection{Left preconditioner}
Shustin and Avron~\cite[\S 4.2]{shustin2023riemannian} proposed to endow $\tensM$ with the following metric
\begin{equation}
  \label{eq: CCA metric by Shustin}
  g_{(\matu,\matv)}(\xi,\eta):=\langle\xi_1,\Sigma_{xx}\eta_1\rangle+\langle\xi_2,\Sigma_{yy}\eta_2\rangle\quad\text{for } \xi,\eta\in\tangent_{(\matu,\matv)}\!\tensM,
\end{equation}
where the tangent space $\tangent_{(\matu,\matv)}\!\tensM$ is defined by $\tangent_{(\matu,\matv)}\!\tensM\simeq\tangent_{\matu}\!\tensM_1\times \tangent_{\matv}\!\tensM_2$ and
\begin{equation}\label{eq: tangent space of gen St}
  \tangent_{\matu}\!\tensM_1=\{\matu\matOmega_1+\matu_{\Sigma_{xx}\perp}\matK_1:\matOmega_1\in\mathbb{R}^{m\times m},\matOmega_1^\T=-\matOmega_1,\matK_1\in\mathbb{R}^{(d_x-m)\times m}\}
\end{equation}
is the tangent space of the generalized Stiefel manifold $\tensM_1$ with dimension $md_x-m(m+1)/2$, the matrix $\matu_{\Sigma_{xx}\perp}\in\mathbb{R}^{d_x\times (d_x-m)}$ satisfies that $(\matu_{\Sigma_{xx}\perp})^\T \Sigma_{xx}^{}\matu_{\Sigma_{xx}\perp}=\matI_{d_x-m}$ and $\matu^\T\Sigma_{xx}^{}\matu_{\Sigma_{xx}\perp}=0$. The $\tangent_{\matv}\!\tensM_2$ is defined in a same fashion.

In our framework, it is equivalent that the operators in~\cref{eq: preconditioned metric on product space} are defined by $\bartensH_1(\matu,\matv)[\eta_1]=\Sigma_{xx}\eta_1$ and $\bartensH_2(\matu,\matv)[\eta_2]=\Sigma_{yy}\eta_2$, which have left preconditioning effect. The orthogonal projection with respect to $g$ of a vector $\bareta\in\tangent_{(\matu,\matv)}\!\tensE\simeq\tensE$ onto $\tangent_{(\matu,\matv)}\!\tensM$ is given by $\Pi_{g,(\matu,\matv)}(\bareta)=\left(\bareta_1-\matu\sym(\matu^\T\Sigma_{xx}\bareta_1),\bareta_2-\matv\sym(\matv^\T\Sigma_{yy}\bareta_2)\right)$, where $\tensE=\mathbb{R}^{d_x\times m}\times\mathbb{R}^{d_y\times m}$ is the ambient space of $\tensM$. Therefore, it follows from~\cref{eq: Riemannian gradient on product space} that the Riemannian gradient is 
\begin{equation}
  \label{eq: CCA Riemannian gradient under Shustin}
  \begin{aligned}
    \grad_g f(\matu,\matv)=(&-\Sigma_{xx}^{-1}\Sigma_{xy}\matv\matN+\matu\sym(\matu^\T\Sigma_{xy}\matv\matN),\\
    &-\Sigma_{yy}^{-1}\Sigma_{xy}^\T\matu\matN+\matv\sym(\matv^\T\Sigma_{xy}^\T\matu\matN)).
  \end{aligned}
\end{equation}

Since the local convergence rate of Riemannian optimization methods is closely related to the condition number $\kappa_g(\Hess_g\!f(\matu^*,\matv^*))$ (see~\cref{subsec: local convergence}), we first compute the Riemannian Hessian of $f$ at $(\matu,\matv)$ along $\eta=(\eta_1,\eta_2)\in\tangent_{(\matu,\matv)}\!\tensM$ by using~\cref{eq: Riemannian Hessian}
\begin{equation}
  \label{eq: Riemannian Hessian CCA Shustin}
  \begin{aligned}
    \Hess_g\!f(\matu,\matv)[\eta]=\Pi_{g,(\matu,\matv)}\big(&\eta_1\sym(\matu^\T\Sigma_{xy}\matv\matN)+\matu\sym(\eta_1^\T\Sigma_{xy}\matv\matN)\\
    &+\matu\sym(\matu^\T\Sigma_{xy}\eta_2\matN)-\Sigma_{xx}^{-1}\Sigma_{xy}\eta_2\matN,\\
    &\eta_2\sym(\matv^\T\Sigma_{xy}^\T\matu\matN)+\matv\sym(\eta_2^\T\Sigma_{xy}^\T\matu\matN)\\
    &+\matv\sym(\matv^\T\Sigma_{xy}^\T\eta_1\matN)-\Sigma_{yy}^{-1}\Sigma_{xy}^\T\eta_1\matN\big),
  \end{aligned}
\end{equation}
Then, the condition number of $\Hess_g\!f(\matu^*,\matv^*)$ can be computed as follows. 
\begin{proposition}\label{prop: condition number of Shustin}
  Let $\sigma_1>\sigma_2>\cdots>\sigma_{m+1}\geq\cdots\geq\sigma_{\min\{d_x,d_y\}}$ be singular values of the matrix $\Sigma_{xx}^{-1/2}\Sigma_{xy}\Sigma_{yy}^{-1/2}$. It holds that
  \[\kappa_g(\Hess_g\!f(\matu^*,\matv^*))=\frac{\max\left\{\frac12(\mu_1+\mu_2)(\sigma_1+\sigma_2),\mu_1(\sigma_1+\sigma_{m+1})\right\}}{\min\{\min_{i,j\in[m],i\neq j}\frac12(\mu_i-\mu_j)(\sigma_i-\sigma_j),\mu_m(\sigma_m-\sigma_{m+1})\}},\]where $[m]:=\{1,2,\dots,m\}$.
\end{proposition}

A proof is given in~\cite[Theorem 5]{shustin2021faster}. Nevertheless, we provide a slightly different proof for~\cref{prop: condition number of Shustin} in~\cref{app: proof} to facilitate the proofs of~\cref{prop: condition number of CCA in new metric,prop: condition number of SVD}. 
Specifically, the proof sketch of~\cref{prop: condition number of Shustin,prop: condition number of CCA in new metric,prop: condition number of SVD}  follows from the same procedure: 1) compute the Riemannian Hessian under the given metric $\Hess_{g}\!f(\matu^*,\matv^*)[\eta]$; 2) compute the Rayleigh quotient $q(\eta)$; 3) compute the maximum and minimum of Rayleigh quotient by taking the parametrization of tangent spaces into $q(\eta)$. Note that~\cref{prop: condition number of Shustin} boils down to
$\kappa_g(\Hess_g\!f(\matu^*,\matv^*))=(\sigma_1+\sigma_2)/(\sigma_1-\sigma_2)$
for $m=1$, which coincides with the result in~\cite[Lemma 4.1]{shustin2023riemannian}.

\subsection{New left and right preconditioner} 
Observing from the second-order information in~\cref{eq: Riemannian Hessian CCA Shustin}, we aim to approximate the diagonal blocks of~\cref{eq: Riemannian Hessian CCA Shustin} and propose a new metric where the operators in~\cref{eq: preconditioned metric on product space} have both left and right preconditioning effect. To this end, we adopt the left and right preconditioning in~\cref{subsec: left and right} and propose a new Riemannian metric
\begin{equation}
  \label{eq: CCA metric new}
  g_{\mathrm{new},(\matU,\matv)}(\xi,\eta):=\langle\xi_1,\Sigma_{xx}\eta_1\matM_{1,2}\rangle+\langle\xi_2,\Sigma_{yy}\eta_2\matM_{2,2}\rangle,
\end{equation}
where $\matM_{1,2}:=(\sym(\matu^\T\Sigma_{xy}\matv\matN)^2+\delta\matI_m)^{1/2}$, $\matM_{2,2}:=(\sym(\matv^\T\Sigma_{xy}^\T\matu\matN)^2+\delta\matI_m)^{1/2}$, and $\delta>0$. The projection operator $\Pi_{\mathrm{new},(\matu,\matv)}$ is given by the following proposition.
\begin{proposition}
  \label{prop: orthogonal projection operator of new metric}
  Given the new metric~\cref{eq: CCA metric new}, the orthogonal projection operator on $\tangent_{(\matU,\matv)}\!\tensM$ is given by
  \begin{equation}\label{eq: CCA projection}
    \Pi_{\mathrm{new},(\matu,\matv)}(\bareta)=(\Pi_{\mathrm{new},\matu}(\bareta_1),\Pi_{\mathrm{new},\matv}(\bareta_2))=(\bareta_1-\matu\matS_1^{}\matM_{1,2}^{-1},\bareta_2-\matv\matS_2^{}\matM_{2,2}^{-1})
  \end{equation}
  for $\bareta\in\tangent_{(\matu,\matv)}\!\tensE\simeq\tensE$,
  where $\matS_1$, $\matS_2$ are the unique solutions of the Lyapunov equations $\matM_{1,2}^{-1}\matS_1^{}+\matS_1^{}\matM_{1,2}^{-1}=2\sym(\matu^\T\Sigma_{xx}\bareta_1)$ and $\matM_{2,2}^{-1}\matS_2^{}+\matS_2^{}\matM_{2,2}^{-1}=2\sym(\matv^\T\Sigma_{yy}\bareta_2)$.
\end{proposition}
\begin{proof}
  See~\cref{app: proof of orth proj}.
\end{proof}

It follows from~\cref{prop: Riemannian gradient as Riemannian submanifold of product space,prop: orthogonal projection operator of new metric} that the Riemannian gradient of $f$ at $(\matu,\matv)\in\tensM$ is 
  \begin{equation}
    \label{eq: Riemannian gradient under CCA new metric}
    \grad_{\mathrm{new}} f(\matu,\matv) = -((\Sigma_{xx}^{-1}\Sigma_{xy}^{}\matv\matN+\matu\matS_1^{})\matM_{1,2}^{-1},(\Sigma_{yy}^{-1}\Sigma_{xy}^\T\matu\matN+\matv\matS_2^{})\matM_{2,2}^{-1}).
  \end{equation}
Since $\matM_{1,2},\matM_{2,2}\in\mathbb{R}^{m\times m}$ and $m\ll\min\{d_x,d_y\}$, the computational cost of the Riemannian gradient under the new metric~\cref{eq: CCA metric new} is comparable to one under~\cref{eq: CCA metric by Shustin}. Subsequently, the Riemannian Hessian of $f$ at $(\matu^*,\matv^*)$ along $\eta$ is given by
\begin{equation*}
  \begin{aligned}
    \Hess_{\mathrm{new}}\!f(\matu^*,\matv^*)[\eta]&=\Pi_{\mathrm{new},(\matu^*,\matv^*)}(\mathrm{D} \bar{G}_{\mathrm{new}}(\matu^*,\matv^*)[\eta]) 
    \\
    =\Pi_{\mathrm{new},(\matu^*,\matv^*)}(&-\Sigma_{xx}^{-1}\Sigma_{xy}^{}\eta_2\matN\matM_{1,2}^{-1}+\Sigma_{xx}^{-1}\Sigma_{xy}\matv^*\matN\matM_{1,2}^{-1}\dot{\matM}_{1,2}^{}\matM_{1,2}^{-1}\\
    &-\eta_1\matS_1^{}\matM_{1,2}^{-1}-\matu^*\dot{\matS}_1^{}\matM_{1,2}^{-1}+\matu^*\matS_1^{}\matM_{1,2}^{-1}\dot{\matM}_{1,2}^{}\matM_{1,2}^{-1},\\
    &-\Sigma_{yy}^{-1}\Sigma_{xy}^\T\eta_1\matN\matM_{2,2}^{-1}+\Sigma_{yy}^{-1}\Sigma_{xy}^\T\matu^*\matN\matM_{2,2}^{-1}\dot{\matM}_{2,2}^{}\matM_{2,2}^{-1}\\
    &-\eta_2\matS_2^{}\matM_{2,2}^{-1}-\matv^*\dot{\matS}_2\matM_{2,2}^{-1}+\matv^*\matS_2^{}\matM_{2,2}^{-1}\dot{\matM}_{2,2}\matM_{2,2}^{-1})
  \end{aligned}
\end{equation*}
where $\bar{G}_{\mathrm{new}}:\tensE\to\mathbb{R}$ is a smooth extension of $\grad_{\mathrm{new}}f$, $\dot{\matM}_{1,2}:=\mathrm{D}\matM_{1,2}(\matu^*,\matv^*)[\eta]$, $\dot{\matM}_{2,2}:=\mathrm{D}\matM_{2,2}(\matu^*,\matv^*)[\eta]$, and the symmetric matrices $\dot{\matS}_1$ and $\dot{\matS}_2$ satisfy the Lyapunov equations
\begin{eqnarray*}
  \sym(\dot{\matM}_{1,2}\matS_1+{\matM}_{1,2}\dot{\matS}_1+\dot{\matM}_{1,2}\Sigma\matN+{\matM}_{1,2}(\eta_1^\T\Sigma_{xy}\matv^*+(\matu^*)^\T\Sigma_{xy}\eta_2))&=0,\\
  \sym(\dot{\matM}_{2,2}\matS_2+{\matM}_{2,2}\dot{\matS}_2+\dot{\matM}_{2,2}\Sigma\matN+{\matM}_{2,2}(\eta_2^\T\Sigma_{xy}^\T\matu^*+(\matv^*)^\T\Sigma_{xy}^\T\eta_1))&=0.
\end{eqnarray*}

Finally, we illustrate the effect of the metric~\cref{eq: CCA metric new} by computing the condition number of $\Hess_g\!f(\matu^*,\matv^*)$ in the following proposition, which can be proved in a similar fashion as in~\cref{prop: condition number of Shustin}.
\begin{proposition}\label{prop: condition number of CCA in new metric}
  Let $\sigma_1>\sigma_2>\cdots>\sigma_{m+1}\geq\cdots\geq\sigma_{\min\{d_x,d_y\}}$ be the singular values of the matrix $\Sigma_{xx}^{-1/2}\Sigma_{xy}\Sigma_{yy}^{-1/2}$. Then, the condition number at the local minimizer $(\matu^*,\matv^*)$ is computed by
  \begin{equation}\label{eq: new condition number CCA}
    \kappa_\mathrm{new}(\Hess_\mathrm{new}\!f(\matu^*,\matv^*))=\frac{\max\{\underset{i,j\in[m],i\neq j}{\max}\frac{(\mu_i+\mu_j)(\sigma_i+\sigma_j)}{\sqrt{\mu_i^2\sigma_i^2+\delta}+\sqrt{\mu_j^2\sigma_j^2+\delta}},\underset{i\in[m]}{\max}\frac{\mu_i(\sigma_i+\sigma_{m+1})}{\sqrt{\mu_i^2\sigma_i^2+\delta}}\}}{\min\{\underset{i,j\in[m],i\neq j}{\min}\frac{(\mu_i-\mu_j)(\sigma_i-\sigma_j)}{\sqrt{\mu_i^2\sigma_i^2+\delta}+\sqrt{\mu_j^2\sigma_j^2+\delta}},\underset{i\in[m]}{\min}\frac{\mu_i(\sigma_i-\sigma_{m+1})}{\sqrt{\mu_i^2\sigma_i^2+\delta}}\}}.
  \end{equation}
\end{proposition}
\begin{proof}
  By following the same procedure as~\cref{prop: condition number of Shustin}, we compute the Rayleigh quotient~\cref{eq: Rayleigh quotient} under the proposed metric~\cref{eq: CCA metric new}, and evaluate its upper and lower bounds. To this end, we firstly compute and simplify the Riemannian Hessian $\Hess_{\mathrm{new}}\!f(\matu^*,\matv^*)[\eta]$ as follows.

  By solving the Lyapunov equations in~\cref{prop: orthogonal projection operator of new metric}, we obtain that $\matS_1=\matS_2=-\Sigma\matN$, where $\Sigma=\diag(\sigma_1,\sigma_2,\dots,\sigma_m)$. Since $(\matu^*,\matv^*)$ is a critical point of $f$, it follows from~\cref{prop: diff metric} that $\grad_\mathrm{new} f(\matu^*,\matv^*)=0$ and thus $\Sigma_{xx}^{-1}\Sigma_{xy}^{}\matv^*=\matu^*\Sigma$ and $\Sigma_{yy}^{-1}\Sigma_{xy}^\T\matu^*=\matv^*\Sigma$. Hence, we can simplify the Riemannian Hessian $\Hess_{\mathrm{new}}\!f(\matu^*,\matv^*)[\eta]$ to
  \begin{equation*}
    \begin{aligned}
      \Hess_{\mathrm{new}}\!f(\matu^*,\matv^*)[\eta]&=\Pi_{\mathrm{new},(\matu^*,\matv^*)}(\mathrm{D} \bar{G}_{\mathrm{new}}(\matu^*,\matv^*)[\eta]) 
      \\
      &=\Pi_{\mathrm{new},(\matu^*,\matv^*)}(-\Sigma_{xx}^{-1}\Sigma_{xy}^{}\eta_2^{}\matN(\matM_{1,2}^*)^{-1}+\eta_1^{}\Sigma\matN(\matM_{1,2}^*)^{-1},\\
      &\qquad\qquad\qquad~~~-\!\Sigma_{yy}^{-1}\Sigma_{xy}^\T\eta_1^{}\matN(\matM_{2,2}^*)^{-1}+\eta_2^{}\Sigma\matN(\matM_{2,2}^*)^{-1}),
    \end{aligned}
  \end{equation*}
  where $\matM_{1,2}^*=\matM_{2,2}^*=(\Sigma^2\matN^2+\delta\matI_m)^{1/2}$ are diagonal matrices, and we use the characterization of $(\tangent_{\matu^*}\!\tensM_1)^\perp$ and $(\tangent_{\matv^*}\!\tensM_2)^\perp$ in~\cref{eq: normal space to new metric}.

  Subsequently, we compute the Rayleigh quotient. It follows from~\cref{eq: tangent space of gen St} that
  \begin{equation*}
    \begin{aligned}
      q(\eta)&=\frac{g_{\mathrm{new},(\matu^*,\matv^*)}(\eta,\Hess_{\mathrm{new}}\!f(\matu^*,\matv^*)[\eta])}{g_{\mathrm{new},(\matu^*,\matv^*)}(\eta,\eta)}\\
      &=\frac{\langle\eta_1,\Sigma_{xx}\eta_1\Sigma\matN\rangle-2\langle\eta_1,\Sigma_{xy}\eta_2\matN\rangle+\langle\eta_2,\Sigma_{yy}\eta_2\Sigma\matN\rangle}{\langle\eta_1,\Sigma_{xx}\eta_1\matM_{1,2}^*\rangle+\langle\eta_2,\Sigma_{yy}\eta_2\matM_{2,2}^*\rangle}
    \end{aligned}
  \end{equation*}
  for $\eta=(\eta_1,\eta_2)\in\tangent_{(\matu^*,\matv^*)}\!\tensM$, where we use the facts that $\dot{\matS}_1$ and $\dot{\matS}_2$ are symmetric and $\langle\eta_1,\Sigma_{xx}\matu^*\dot{\matS}_1\rangle=\langle\eta_2,\Sigma_{yy}\matv^*\dot{\matS}_2\rangle=0$ by~\cref{eq: tangent space of gen St}. 
  
  We observe that only the denominator of $q(\eta)$ is different from the denominator $\langle\eta_1,\Sigma_{xx}\eta_1\rangle+\langle\eta_2,\Sigma_{yy}\eta_2\rangle$ in~\cref{eq: Rayleigh quotient of CCA}, and $\matM_{1,2}^*,\matM_{2,2}^*$ are diagonal matrices. Consequently, we can evaluate the upper and lower bounds of $q(\eta)$ in a similar fashion as~\cref{app: proof}.
\end{proof}

\paragraph{Improved local convergence rate}
We illustrate the effect of the proposed Riemannian metric~\cref{eq: CCA metric new} through the improved condition number, which is able to accelerate the Riemannian methods in the sense of~\cref{thm: local convergence}. To this end, we first provide~\cref{lem: sij} to simplify~\cref{eq: new condition number CCA}. Next, we prove that $\kappa_\mathrm{new}(\Hess_\mathrm{new}\!f(\matu^*,\matv^*))\leq\kappa_g(\Hess_g\!f(\matu^*,\matv^*))$ in~\cref{prop: better condition number}.

\begin{lemma}\label{lem: sij}
  Denote 1) $\bar{v}_{ij}(\delta):=(\mu_i+\mu_j)(\sigma_i+\sigma_j)/(\sqrt{\mu_i^2\sigma_i^2+\delta}+\sqrt{\mu_j^2\sigma_j^2+\delta})$ and $\bar{v}_{i,m+1}(\delta)=\mu_i(\sigma_i+\sigma_{m+1})/\sqrt{\mu_i^2\sigma_i^2+\delta}$; 2) $\underline{v}_{ij}(\delta):=(\mu_i-\mu_j)(\sigma_i-\sigma_j)/(\sqrt{\mu_i^2\sigma_i^2+\delta}+\sqrt{\mu_j^2\sigma_j^2+\delta})$ and $\underline{v}_{i,m+1}(\delta):=\mu_i(\sigma_i-\sigma_{m+1})/\sqrt{\mu_i^2\sigma_i^2+\delta}$ for $i,j\in[m]$.  It holds that 
  \[\bar{v}_{ij}(0)>\bar{v}_{ik}(0)\qquad\text{and}\qquad\underline{v}_{ij}(0)<\underline{v}_{ik}(0)\]
  for all $1\leq i<j<k\leq m+1$ with $m\geq 3$.
\end{lemma}
\begin{proof}
  We observe that $\bar{v}_{ij}(0)=(\mu_i+\mu_j)(\sigma_i+\sigma_j)/(\mu_i\sigma_i+\mu_j\sigma_j)$ and $\underline{v}_{ij}(0)=(\mu_i-\mu_j)(\sigma_i-\sigma_j)/(\mu_i\sigma_i+\mu_j\sigma_j)$ for $i\in[m]$ and $j\in[m+1]$ with $\mu_{m+1}=0$. First, we prove that $\bar{v}_{ij}(0)>\bar{v}_{ik}(0)$. Since $\mu_i>\mu_j>\mu_k$ and $\sigma_i>\sigma_j>\sigma_k$, it holds that 
  \begin{equation*}
    \begin{aligned}
      \bar{v}_{ij}(0)-\bar{v}_{ik}(0)&=\frac{(\mu_i\sigma_j+\mu_j\sigma_i)(\mu_i\sigma_i+\mu_k\sigma_k)-(\mu_i\sigma_k+\mu_k\sigma_i)(\mu_i\sigma_i+\mu_j\sigma_j)}{(\mu_i\sigma_i+\mu_j\sigma_j)(\mu_i\sigma_i+\mu_k\sigma_k)}\\
      &=\frac{(\mu_i^2-\mu_j\mu_k)\sigma_i(\sigma_j-\sigma_k)+\mu_i(\mu_j-\mu_k)(\sigma_i^2-\sigma_j\sigma_k)}{(\mu_i\sigma_i+\mu_j\sigma_j)(\mu_i\sigma_i+\mu_k\sigma_k)}>0.
    \end{aligned}
  \end{equation*}
  Therefore, $\bar{v}_{ij}(0)>\bar{v}_{ik}(0)$ holds. The results $\underline{v}_{ij}(0)<\underline{v}_{ik}(0)$ is ready by using $\bar{v}_{ij}(0)+\underline{v}_{ij}(0)=2$ and $\bar{v}_{ik}(0)+\underline{v}_{ik}(0)=2$.
\end{proof}

Then, it follows from~\cref{lem: sij} and the continuity of $\bar{v}_{ij}$ and $\underline{v}_{ij}$ with respect to $\delta\in[0,\infty)$ that there exists a constant $\bar{\delta}_1>0$, such that 
\[\bar{v}_{ij}(\delta)>\bar{v}_{ik}(\delta)\qquad\text{and}\qquad\underline{v}_{ij}(\delta)<\underline{v}_{ik}(\delta)\]
hold for all $1\leq i<j<k\leq m+1$ and $\delta\in(0,\bar{\delta}_1)$. Therefore, we can simplify the condition number
\begin{align}
  \kappa_\mathrm{new}(\Hess_\mathrm{new}\!f(\matu^*,\matv^*))&=\frac{\max_{i\in[m],j\in[m+1],i\neq j}\bar{v}_{ij}(\delta)}{\min_{i\in[m],j\in[m+1],i\neq j}\underline{v}_{ij}(\delta)}=\frac{\max_{i\in[m]}\bar{v}_{i,i+1}(\delta)}{\min_{i\in[m]}\underline{v}_{i,i+1}(\delta)}.
  \label{eq: simplified condition number}
\end{align}
We aim to prove that $\kappa_\mathrm{new}(\Hess_\mathrm{new}\!f(\matu^*,\matv^*))\leq\kappa_g(\Hess_g\!f(\matu^*,\matv^*))$ for $m\geq 2$. Note that $\kappa_\mathrm{new}(\Hess_\mathrm{new}\!f(\matu^*,\matv^*))=\kappa_g(\Hess_g\!f(\matu^*,\matv^*))$ for $m=1$ since the right preconditioners in~\cref{eq: CCA metric new} boil down to scalars that no longer have preconditioning effect.
\begin{proposition}\label{prop: better condition number}
  Assume that $m\geq 2$. There exists a constant $\bar{\delta}>0$, such that 
  \[\kappa_\mathrm{new}(\Hess_\mathrm{new}\!f(\matu^*,\matv^*))\leq\kappa_g(\Hess_g\!f(\matu^*,\matv^*))\]
  holds for all $\delta\in(0,\bar{\delta})$ in~\cref{eq: CCA metric new}.
\end{proposition}
\begin{proof}
  If $\argmax_{i\in[m]}\bar{v}_{i,i+1}(0)=\{i^*\}$ for some $i^*\in[m]$, it follows from $\bar{v}_{i,i+1}(0)+\underline{v}_{i,i+1}(0)=2$ that $\argmin_{i\in[m]}\underline{v}_{i,i+1}(0)=\argmax_{i\in[m]}\bar{v}_{i,i+1}(0)=\{i^*\}$. Since $\bar{v}_{ij}(\delta)$ and $\underline{v}_{ij}(\delta)$ are continuous, there exists $\bar{\delta}>0$, such that $\{i^*\}=\argmax_{i\in[m]}\bar{v}_{i,i+1}(\delta)$ and $\{i^*\}=\argmin_{i\in[m]}\underline{v}_{i,i+1}(\delta)$ for all $\delta\in[0,\bar{\delta})$. Subsequently, we obtain that
    \begin{align}
      \kappa_\mathrm{new}(\Hess_\mathrm{new}\!f(\matu^*,\matv^*))&=\frac{\bar{v}_{i^*,i^*+1}(\delta)}{\underline{v}_{i^*,i^*+1}(\delta)}=\frac{\frac12(\mu_{i^*}+\mu_{i^*+1})(\sigma_{i^*}+\sigma_{i^*+1})}{\frac12(\mu_{i^*}-\mu_{i^*+1})(\sigma_{i^*}-\sigma_{i^*+1})}\nonumber\\
      &\leq\frac{\max\{\frac12(\mu_1+\mu_2)(\sigma_1+\sigma_2),\mu_1(\sigma_1+\sigma_{m+1})\}}{\min\{\underset{i,j\in[m],i\neq j}{\min}\frac12(\mu_i-\mu_j)(\sigma_i-\sigma_j),\mu_m(\sigma_m-\sigma_{m+1})\}}\nonumber\\
      &=\kappa_g(\Hess_g\!f(\matu^*,\matv^*)).\nonumber
    \end{align}

  If $\{i_1^*,i_2^*\}\subseteq\argmax_{i\in[m]}\bar{v}_{i,i+1}(0)$ for $i_1^*<i_2^*$, it follows from the continuity of $\bar{v}_{ij}$ and $\underline{v}_{ij}$ that there exists $\bar{\delta}_{i_1^*,i_2^*}>0$, such that $\bar{v}_{i_1^*,i_1^*+1}(\delta)>\bar{v}_{i_2^*,i_2^*+1}(\delta)$ and $\underline{v}_{i_1^*,i_1^*+1}(\delta)>\underline{v}_{i_2^*,i_2^*+1}(\delta)$. The rationales behind these are $\bar{v}_{i_1^*,i_1^*+1}(0)=\bar{v}_{i_2^*,i_2^*+1}(0)$, $\underline{v}_{i_1^*,i_1^*+1}(0)=\underline{v}_{i_2^*,i_2^*+1}(0)$, and the derivatives satisfy
  \begin{equation*}
    \begin{aligned}
      \bar{v}_{i_1^*,i_1^*+1}^\prime(0)&=-\frac{\bar{v}_{i_1^*,i_1^*+1}(0)}{2\mu_{i_1^*}\sigma_{i_1^*}\mu_{i_1^*+1}\sigma_{i_1^*+1}}>-\frac{\bar{v}_{i_2^*,i_2^*+1}(0)}{2\mu_{i_2^*}\sigma_{i_2^*}\mu_{i_2^*+1}\sigma_{i_2^*+1}}=\bar{v}_{i_2^*,i_2^*+1}^\prime(0),\\
      \underline{v}_{i_1^*,i_1^*+1}^\prime(0)&=-\frac{\underline{v}_{i_1^*,i_1^*+1}(0)}{2\mu_{i_1^*}\sigma_{i_1^*}\mu_{i_1^*+1}\sigma_{i_1^*+1}}>-\frac{\underline{v}_{i_2^*,i_2^*+1}(0)}{2\mu_{i_2^*}\sigma_{i_2^*}\mu_{i_2^*+1}\sigma_{i_2^*+1}}=\underline{v}_{i_2^*,i_2^*+1}^\prime(0)
    \end{aligned}
  \end{equation*}
  if $i_2^*<m$ and
  \begin{equation*}
    \begin{aligned}
      \bar{v}_{i_1^*,i_1^*+1}^\prime(0)&=-\frac{\bar{v}_{i_1^*,i_1^*+1}(0)}{2\mu_{i_1^*}\sigma_{i_1^*}\mu_{i_1^*+1}\sigma_{i_1^*+1}}>-\frac{\bar{v}_{m,m+1}(0)}{2\mu_{m}^2\sigma_{m}^2}=\bar{v}_{m,m+1}^\prime(0),\\
      \underline{v}_{i_1^*,i_1^*+1}^\prime(0)&=-\frac{\underline{v}_{i_1^*,i_1^*+1}(0)}{2\mu_{i_1^*}\sigma_{i_1^*}\mu_{i_1^*+1}\sigma_{i_1^*+1}}>-\frac{\underline{v}_{m,m+1}(0)}{2\mu_{m}^2\sigma_{m}^2}=\underline{v}_{m,m+1}^\prime(0)
    \end{aligned}
  \end{equation*}
  if $i_2^*=m$. Therefore, there exists $\bar{\delta}\in(0,\min\{\bar{\delta}_{i,j}:i,j\in\argmax_{i\in[m]}\bar{v}_{i,i+1}(0)\})$ and $i^*,j^*\in[m]$, such that: 1) $i^*=\argmax_{i\in[m]}\bar{v}_{i,i+1}(\delta)$; 2) $j^*=\argmin_{i\in[m]}\underline{v}_{i,i+1}(\delta)$ for all $\delta\in[0,\bar{\delta})$; 3) $i^*<j^*$. Consequently, we obtain that 
  \begin{align}
    &~~~~\kappa_\mathrm{new}(\Hess_\mathrm{new}\!f(\matu^*,\matv^*))=\frac{\bar{v}_{i^*,i^*+1}(\delta)}{\underline{v}_{j^*,j^*+1}(\delta)}\nonumber\\
    &=\frac{(\mu_{i^*}+\mu_{i^*+1})(\sigma_{i^*}+\sigma_{i^*+1})}{(\mu_{j^*}-\mu_{j^*+1})(\sigma_{j^*}-\sigma_{j^*+1})}\cdot\frac{\sqrt{\mu_{j^*}^2\sigma_{j^*}^2+1}+\sqrt{\mu_{j^*+1}^2\sigma_{j^*+1}^2+1}}{\sqrt{\mu_{i^*}^2\sigma_{i^*}^2+1}+\sqrt{\mu_{i^*+1}^2\sigma_{i^*+1}^2+1}}\nonumber\\
    &<\frac{\max\{\frac12(\mu_1+\mu_2)(\sigma_1+\sigma_2),\mu_1(\sigma_1+\sigma_{m+1})\}}{\min\{\underset{i,j\in[m],i\neq j}{\min}\frac12(\mu_i-\mu_j)(\sigma_i-\sigma_j),\mu_m(\sigma_m-\sigma_{m+1})\}}\nonumber\\
    &=\kappa_g(\Hess_g\!f(\matu^*,\matv^*)).\nonumber
  \end{align}
\end{proof}

It is worth noting that the parameter $\delta>0$ theoretically ensures that~\cref{eq: CCA metric new} is a Riemannian metric. In practice, one can choose a sufficiently small $\delta$, e.g., $\delta=10^{-15}$.

\subsection{RGD and RCG for canonical correlation analysis}
By using the Riemannian metric~\cref{eq: CCA metric new} and required ingredients, we adapt the Riemannian gradient descent (\cref{alg: RGD}) and Riemannian conjugate gradient (\cref{alg: RCG}) methods to solve the CCA problem in~\cref{alg: RGD CCA,alg: RCG CCA}. 

\begin{algorithm}[htbp]
  \caption{RGD for CCA}\label{alg: RGD CCA}
  \begin{algorithmic}[1]
      \REQUIRE $\tensM$ endowed with a metric~\cref{eq: CCA metric new}, initial guess $(\matu^{(0)},\matv^{(0)})\in\tensM$, $t=0$. 
      \WHILE{the stopping criteria are not satisfied}
          \STATE Compute $\eta^{(t)}=-\grad_g f(\matu^{(t)},\matv^{(t)})$ by~\cref{eq: Riemannian gradient under CCA new metric}.
          \STATE Compute the stepsize $s^{(t)}$ by Armijo backtracking~\cref{eq: Armijo backtracking line search}.
          \STATE Update $\matu^{(t+1)}=\Sigma_{xx}^{-1/2}\qf(\Sigma_{xx}^{1/2}(\matu^{(t)}+\eta_1^{(t)}))$, $\matv^{(t+1)}=\Sigma_{yy}^{-1/2}\qf(\Sigma_{yy}^{1/2}(\matv^{(t)}+\eta_2^{(t)}))$; $t= t+1$.
      \ENDWHILE
      \ENSURE $(\matu^{(t)},\matv^{(t)})\in\tensM$. 
  \end{algorithmic}
\end{algorithm}
\begin{algorithm}[htbp]
  \caption{RCG for CCA}\label{alg: RCG CCA}
  \begin{algorithmic}[1]
      \REQUIRE $\tensM$ endowed with a metric~\cref{eq: CCA metric new}, initial guess $(\matu^{(0)},\matv^{(0)})\in\tensM$, $t=0$, $\beta^{(0)}=0$. 
      \WHILE{the stopping criteria are not satisfied}
          \STATE Compute $\eta^{(t)}=-\grad_g f(\matu^{(t)},\matu^{(t)})+\beta^{(t)}\Pi_{g,(\matu^{(t)},\matv^{(t)})}(\eta^{(t-1)})$ by~\cref{eq: Riemannian gradient under CCA new metric}. 
          \STATE Compute the stepsize $s^{(t)}$ by Armijo backtracking~\cref{eq: Armijo backtracking line search}.
          \STATE Update $\matu^{(t+1)}=\Sigma_{xx}^{-1/2}\qf(\Sigma_{xx}^{1/2}(\matu^{(t)}+\eta_1^{(t)}))$, $\matv^{(t+1)}=\Sigma_{yy}^{-1/2}\qf(\Sigma_{yy}^{1/2}(\matv^{(t)}+\eta_2^{(t)}))$; $t= t+1$.
      \ENDWHILE
      \ENSURE $(\matu^{(t)},\matv^{(t)})\in\tensM$. 
  \end{algorithmic}
\end{algorithm}

Note that 1) the retraction mapping is the \emph{generalized QR factorization}~\cite{sato2019cholesky} with respect to $\Sigma_{xx}$ and $\Sigma_{yy}$, i.e., 
\[\retr_{(\matu,\matv)}(\eta):=(\Sigma_{xx}^{-\frac12}\qf(\Sigma_{xx}^{\frac12}(\matu+\eta_1^{})),\Sigma_{yy}^{-\frac12}\qf(\Sigma_{yy}^{\frac12}(\matv+\eta_2^{})))\quad\text{for }\eta\in\tangent_{(\matu,\matv)}\!\tensM,\]
where $\qf(\matx)$ refers to the $\matQ$ factor in the QR factorization $\matQ\matR=\matx$.
In practice, the retraction can be efficiently computed~\cite{sato2019cholesky} by $\retr_{(\matu,\matv)}(\eta)=((\matu+\eta_1^{})\matR_1^{-1},(\matv+\eta_2^{})\matR_2^{-1})$ instead, where $\matR_1^\T\matR_1^{}=(\matu+\eta_1^{})^\T\Sigma_{xx}^{}(\matu+\eta_1^{})$ and $\matR_2^\T\matR_2^{}=(\matv+\eta_2^{})^\T\Sigma_{yy}^{}(\matv+\eta_2^{})$ are Cholesky factorization; 
2) the vector transport in~\cref{alg: RCG CCA} is defined by the projection operator in~\cref{eq: CCA projection}, i.e., $\tensT_{t\gets t-1}(\eta)=\Pi_{g,(\matu^{(t)},\matv^{(t)})}(\eta)$ for $\eta\in\tangent_{(\matu^{(t-1)},\matv^{(t-1)})}\!\tensM$.

\subsection{Numerical validation}\label{subsec: numerical CCA}
\cref{alg: RGD CCA,alg: RCG CCA} are implemented in toolbox Manopt v7.1.0~\cite{boumal2014manopt}, a Matlab library for Riemannian methods. The stopping criteria are the same as default settings in Manopt. All experiments are performed on a MacBook Pro 2019 with MacOS Ventura 13.3, 2.4 GHz 8 core Intel Core i9 processor, 32GB memory, and Matlab R2020b. The codes are available at~\url{https://github.com/JimmyPeng1998/popman}.

\begin{table}[htbp]
  \centering
  \footnotesize
  \caption{Compared metrics in CCA.}
  \label{tab: CCA compared metrics}
  \begin{tabular}{crrrrr}
    \toprule
     & (E) & (L1) & (L2) & (L12) & (LR12)\\
    \midrule
    $\bartensH_{1}(\matu,\matv)[\eta_1]$ & $\eta_1$ & $\Sigma_{xx}\eta_1$ & $\eta_1$ & $\Sigma_{xx}\eta_1$ & $\Sigma_{xx}\eta_1\matM_{1,2}$\\
    $\bartensH_{2}(\matu,\matv)[\eta_2]$ & $\eta_2$ & $\eta_2$ & $\Sigma_{yy}\eta_2$ & $\Sigma_{yy}\eta_2$ & $\Sigma_{yy}\eta_2\matM_{2,2}$\\
    \bottomrule
  \end{tabular}
\end{table}

We test the performance of RGD and RCG under different metrics, i.e., five different choice of $\bartensH_1,\bartensH_2$ in $g_{(\matu,\matv)}(\xi,\eta)=\langle\xi_1, \bartensH_{1}(\matu,\matv)[\eta_1]\rangle+\langle\xi_2, \bartensH_{2}(\matu,\matv)[\eta_2]\rangle$;
see~\cref{tab: CCA compared metrics}. The Euclidean metric is denoted by ``(E)''. ``(L1)'' and ``(L2)'' are the metrics wherein only one component of $\tensM=\tensM_1\times\tensM_2$ is endowed with a preconditioned metric. The metric~\cref{eq: CCA metric by Shustin} proposed by~\cite{shustin2023riemannian} is called ``(L12)''.
The metric~\cref{eq: CCA metric new}, denoted by ``(LR12)'', has the effect of preconditioning both on the left and right. We set $d_x=800$, $d_y=400$, $n=30000$, $m=5$, $\delta=10^{-15}$, $\lambda_x=\lambda_y=10^{-6}$, and $\matN=\diag(m,m-1,\dots,1)$. Elements of the data matrices $\matx$ and $\matY$ are sampled i.i.d. from the uniform distribution on~$[0,1]$. The performance of a method is evaluated by the residual $(f(\matu,\matv)-f_{\min})$, gradient norm ``gnorm'', and the subspace distances $D(\matu,\matu^*):=\|\matu\matu^\T-\matu^*(\matu^*)^\T\|_\mathrm{F}$ and $D(\matv,\matv^*):=\|\matv\matv^\T-\matv^*(\matv^*)^\T\|_\mathrm{F}$, where $f_{\min}=f(\matu^*,\matv^*)$ and $(\matu^*,\matv^*)$ is defined in~\cref{eq: cca solution}.

Numerical results are reported in~\cref{fig: CCA synthetic results,fig: CCA computational cost,tab: CCA synthetic results}. We have following observations: 1) the proposed metric~\cref{eq: CCA metric new} improves the performance of RGD and RCG since it benefits more from the second-order information; 2) \cref{fig: CCA computational cost} shows that the computation time per iteration of~\cref{alg: RGD CCA,alg: RCG CCA} is comparable to RGD(L12) and RCG(L12); 3) \cref{tab: CCA synthetic results} illustrates that RGD(LR12) and RCG(LR12) require fewer iterations and less time to reach the stopping criteria than the others. The subspace distances are smaller than $10^{-8}$, and hence the sequences generated by proposed methods converge to the correct subspace.

\begin{figure}[htbp]
    \centering
    \subfigure{\includegraphics[width=0.48\textwidth]{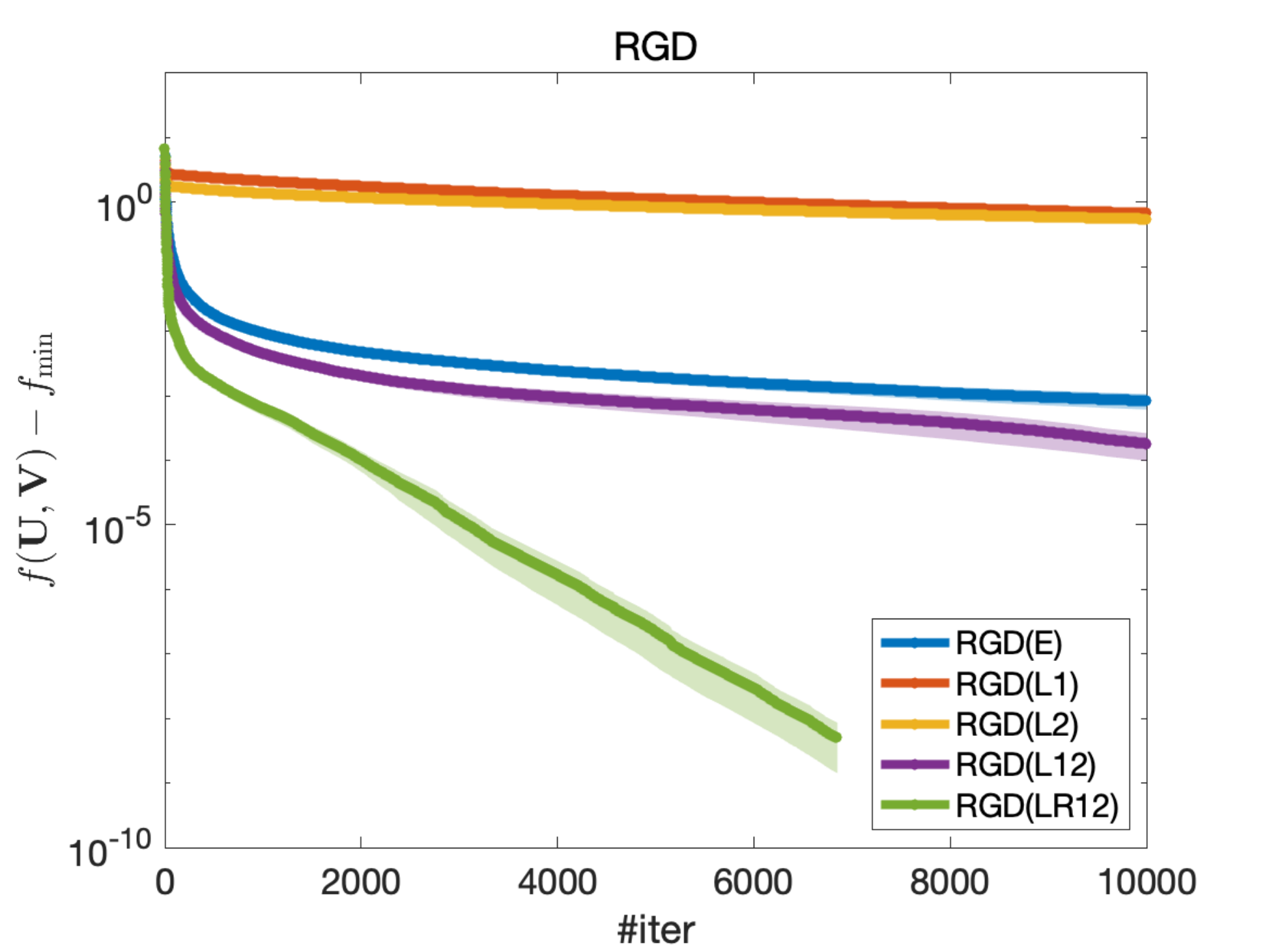}}
    \subfigure{\includegraphics[width=0.48\textwidth]{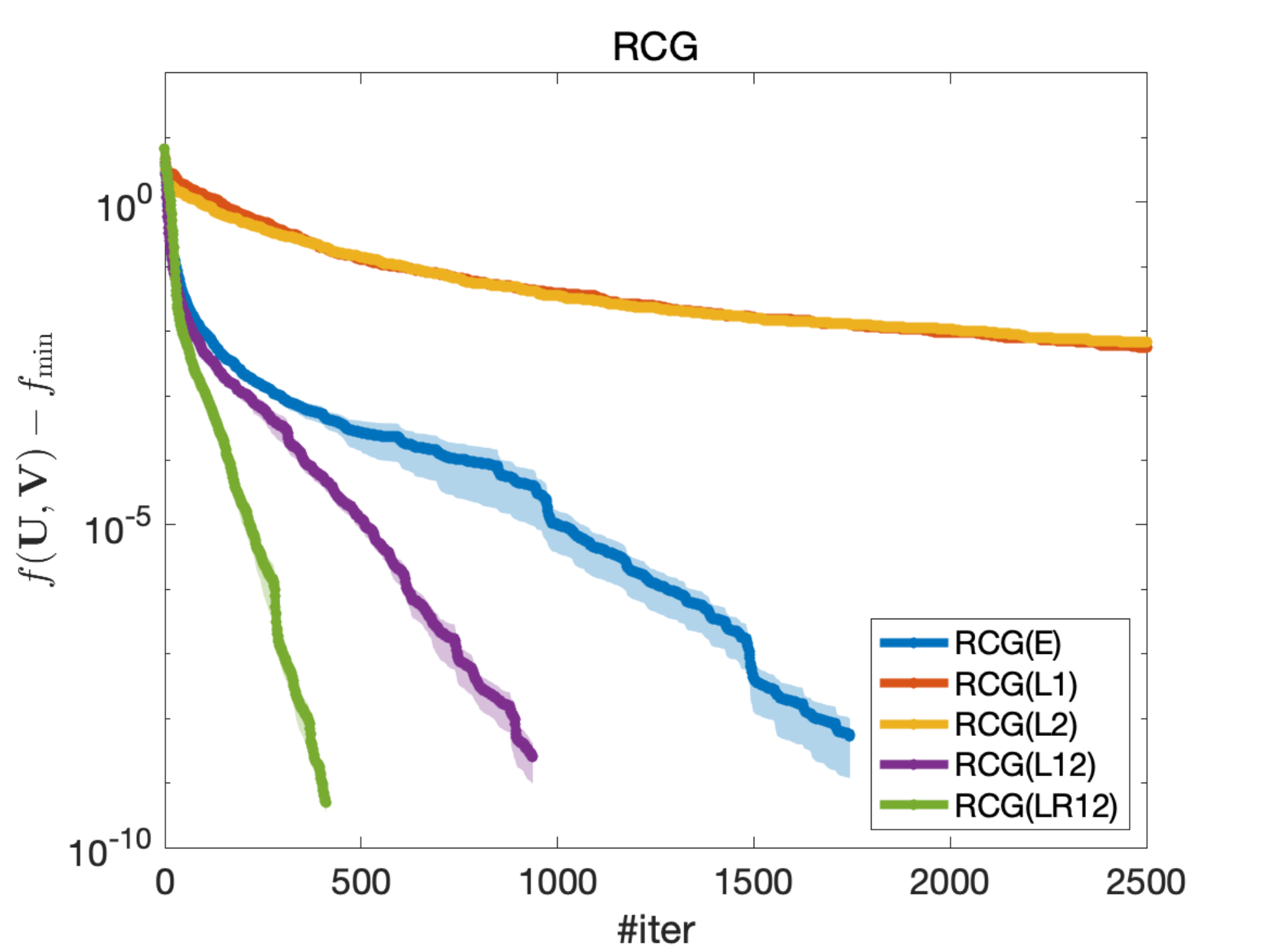}}
    \vspace{-1mm}
    \caption{Numerical results for CCA problem for $d_x=800$, $d_y=400$, and $m=5$. Left: RGD. Right: RCG. Each method is tested for 10 runs.}
    \label{fig: CCA synthetic results}
\end{figure}

\begin{figure}[htbp]
  \centering
  \includegraphics[width=0.9\textwidth]{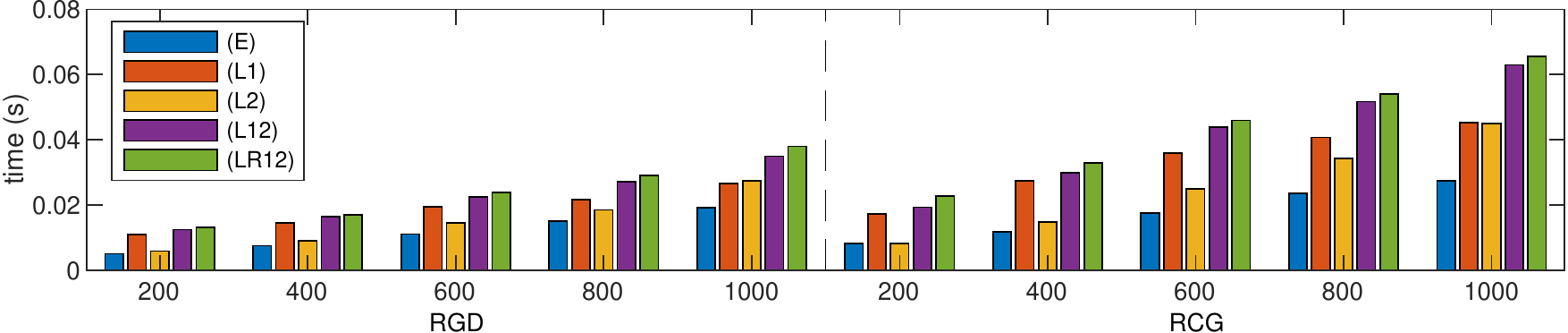}
  \vspace{-1mm}
  \caption{Computation time per iteration for RGD (left) and RCG (right) under different metrics for CCA problem for $d_x=800$, $m=5$, and $d_y=200,400,\dots,1000$.}
  \label{fig: CCA computational cost}
\end{figure}

\begin{table}[htbp]
  \centering
  \footnotesize
  \caption{Convergence results of the CCA problem for $d_x=800$, $d_y=400$, and $m=5$.}
  \label{tab: CCA synthetic results}
  \begin{tabular}{lcrrccccc}
      \toprule
      metric & \multirow{1}*{method} & \multirow{1}*{\#iter} & \multirow{1}*{time (s)} 
      & \multirow{1}*{gnorm} & \multirow{1}*{$D(\matu,\matu^*)$} & \multirow{1}*{$D(\matv,\matv^*)$} & $\kappa_g$\\
      \midrule
      \multirow{2}*{(E)} & RGD & 10000 & 249.11 & 5.95e-02 & 2.69e-05 & 2.66e-05 & \multirow{2}*{2.10e+04}\\
       & RCG & 1745 & 31.03 & 1.70e-05 & 4.01e-10 & 3.89e-10 & \\
      \multirow{2}*{(L1)} & RGD & 10000 & 255.33 & 1.02e+00 & 4.12e-04 & 4.07e-04 & \multirow{2}*{1.43e+07}\\
      & RCG & 2500 & 74.13 & 4.94e-02 & 2.85e-04 & 2.79e-04 & \\
      \multirow{2}*{(L2)} & RGD & 10000 & 245.81 & 8.20e-01 & 4.13e-04 & 4.05e-04 & \multirow{2}*{1.52e+07}\\
       & RCG  & 2500 & 56.16 & 6.90e-02 & 2.93e-04 & 2.90e-04 & \\
      \multirow{2}*{(L12)} & RGD & 10000 & 274.91 & 4.67e-04 & 9.68e-07 & 9.57e-07 & \multirow{2}*{1.12e+04}\\
       & RCG & 937 & 30.39 & 8.82e-07 & 1.68e-09 & 1.65e-09 & \\
      \multirow{2}*{(LR12)} & RGD & 6607 & 195.03 & 1.34e-06 & \textbf{7.47e-16} & \textbf{7.46e-16} & \multirow{2}*{\textbf{2.38e+03}} \\
       & \textbf{RCG} & \textbf{410} & \textbf{15.38} & \textbf{8.49e-07} & 4.63e-09 & 4.53e-09 & \\
      \bottomrule
  \end{tabular}
\end{table}

Moreover, the condition number of the Riemannian Hessian is numerically computed by the Manopt function \texttt{hessianspectrum}: $\kappa_g(\Hess_g\!f(\matu^*,\matv^*))$ of five metrics are $2.10\cdot 10^4$ (E), $1.43\cdot 10^7$ (L1), $1.52\cdot 10^7$ (L2), $1.12\cdot 10^4$ (L12), and $2.38\cdot 10^3$ (LR12). It is direct to verify that these numbers coincide with the theoretical results in~\cref{prop: condition number of Shustin,prop: condition number of CCA in new metric,prop: better condition number}. We observe that the Riemannian Hessian under the proposed metric (LR12) has the smallest condition number among all choices, which is reflected in the observation that RGD(LR12) and RCG(LR12) outperform the others.

\section{Application to truncated singular value decomposition}\label{sec: SVD}
In this section, the truncated singular value decomposition (SVD) problem is considered. Given a matrix $\mata\in\mathbb{R}^{m\times n}$, the $p<\min\{m,n\}$ largest singular vectors $(\matu^*,\matv^*)$ is the global minimizer of the following problem,
\begin{equation}\label{eq: SVD}
  \min\limits_{\matu, \matv}f(\matu,\matv):=-\tr(\matu^\T\matA\matv\matN),\ \subjectto\ (\matu,\matv)\in\tensM:=\St(p,m)\times\St(p,n),
\end{equation}
where $\St(p,m):=\{\matu\in\mathbb{R}^{m\times p}:\matu^\T\matu=\matI_p\}$ is the \emph{Stiefel manifold} and $\matN:=\diag\{\mu_1,\dots,\mu_p\}$ with $\mu_1>\mu_2>\cdots>\mu_p>0$. Sato and Iwai~\cite{sato2013riemannian} proposed RGD and RCG methods to solve problem~\cref{eq: SVD}, where the search space is endowed with the Euclidean metric. We apply the proposed framework to solve~\cref{eq: SVD} by endowing~$\tensM$ with a non-Euclidean metric to accelerate the Riemannian optimization methods. 

\subsection{A new preconditioned metric}
Observe that the Riemannian Hessian of~$f$ at $(\matu,\matv)$ along $\eta=(\eta_1,\eta_2)\in\tangent_{(\matu,\matv)}\!\tensM$ is given by
\begin{equation*}
  \begin{aligned}
    \Hess_{\eucmetric}\!f(\matu,\matv)[\eta]=(&\eta_1\matM_1-\mata\eta_2\matN-\matu\sym(\matu^\T(\eta_1\matM_1-\mata\eta_2\matN)),\\
    &\eta_2\matM_{2}-\mata^\T\eta_1\matN-\matv\sym(\matv^\T(\eta_2\matM_2-\mata^\T\eta_1\matN)))
  \end{aligned}
\end{equation*}
in~\cite[Proposition 3.5]{sato2013riemannian}, where $\matM_1:=\sym(\matu^\T\mata\matv\matN)$ and $\matM_2:=\sym(\matv^\T\mata^\T\matu\matN)$. 
Taking advantage of the diagonal blocks of the Riemannian Hessian and the left and right preconditioning in~\cref{subsec: diagonal approximation}, we define a new preconditioned metric on $\tensM$:
\begin{equation}
  \label{eq: new metric in SVD}
  g_{\mathrm{new},(\matU,\matv)}(\xi,\eta):=\langle\xi_1,\eta_1\matM_{1,2}\rangle+\langle\xi_2,\eta_2\matM_{2,2}\rangle\quad\text{for}\ \xi,\eta\in\tangent_{(\matU,\matv)}\!\tensM, 
\end{equation}
where $\matM_{1,2}=(\sym(\matu^\T\mata\matv\matN)^2+\delta\matI_p)^{1/2}$, $\matM_{2,2}=(\sym(\matv^\T\mata^\T\matu\matN)^2+\delta\matI_p)^{1/2}$,
and $\delta>0$. Note that the left preconditioners are chosen as the identity matrix. The projection operator with respect to~\cref{eq: new metric in SVD} is given by
\begin{equation}
  \Pi_{\mathrm{new},(\matu,\matv)}(\bareta)=(\bareta_1-\matu\matS_1^{}\matM_{1,2}^{-1},\bareta_2-\matv\matS_2^{}\matM_{2,2}^{-1})
  \label{eq: projection operator in SVD}
\end{equation}
for $\bareta\in\tangent_{(\matu,\matv)}\!(\mathbb{R}^{m\times p}\times\mathbb{R}^{n\times p})\simeq \mathbb{R}^{m\times p}\times\mathbb{R}^{n\times p}$, where $\matS_1$, $\matS_2$ are the unique solutions of the Lyapunov equations $\matM_{1,2}^{-1}\matS_1^{}+\matS_1^{}\matM_{1,2}^{-1}=2\sym(\matu^\T\bareta_1)$, $\matM_{2,2}^{-1}\matS_2^{}+\matS_2^{}\matM_{2,2}^{-1}=2\sym(\matv^\T\bareta_2)$. Then, it follows from~\cref{prop: Riemannian gradient as Riemannian submanifold of product space,eq: projection operator in SVD} that 
\begin{equation}\label{eq: SVD gradient g}
  \begin{aligned}
    \grad_{\mathrm{new}} f(\matu,\matv) = ( \mata\matv\matN\matM_{1,2}^{-1}-\matu\matS_1^{}\matM_{1,2}^{-1},\mata^\T\matu\matN\matM_{2,2}^{-1}-\matv\matS_2^{}\matM_{2,2}^{-1}).
  \end{aligned}
\end{equation}
These results can be obtained in a same fashion as in~\cref{prop: orthogonal projection operator of new metric} for CCA. Note that the computational cost of the Riemannian gradient~\cref{eq: SVD gradient g} is comparable to one under the Euclidean metric since $\matM_{1,2},\matM_{2,2}\in\mathbb{R}^{p\times p}$ and $p\ll\min\{m,n\}$. 

The effect of the new metric~\cref{eq: new metric in SVD} is illustrated by the following proposition, which can be proved in a same fashion by letting $\Sigma_{xx}=\matI_{d_x}$, $\Sigma_{yy}=\matI_{d_y}$, $\Sigma_{xy}=\mata$, $d_x=m$ and $d_y=n$ in~\cref{prop: condition number of Shustin,prop: condition number of CCA in new metric}.
\begin{proposition}\label{prop: condition number of SVD}
  Let $\sigma_1>\sigma_2>\cdots>\sigma_p>\sigma_{p+1}\geq\cdots\geq\sigma_{\min\{m,n\}}$ be the singular values of $\mata$,  $\matu^*$ and $\matv^*$ be the $p$ largest left and right singular vectors of $\mata$ respectively. It holds that 
  \begin{equation*}
    \begin{aligned}
      \kappa_\eucmetric(\Hess_\eucmetric\!f(\matu^*,\matv^*)) &=\frac{\max\left\{\frac12(\mu_1+\mu_2)(\sigma_1+\sigma_2),\mu_1(\sigma_1+\sigma_{p+1})\right\}}{\min\{\min_{i,j\in[p],i\neq j}\frac12(\mu_i-\mu_j)(\sigma_i-\sigma_j),\mu_p(\sigma_p-\sigma_{p+1})\}},
      \\
      \kappa_\mathrm{new}(\Hess_\mathrm{new}\!f(\matu^*,\matv^*))&=\frac{\max\{\underset{i,j\in[p],i\neq j}{\max}\frac{(\mu_i+\mu_j)(\sigma_i+\sigma_j)}{\sqrt{\mu_i^2\sigma_i^2+\delta}+\sqrt{\mu_j^2\sigma_j^2+\delta}},\underset{i\in[p]}{\max}\frac{\mu_i(\sigma_i+\sigma_{p+1})}{\sqrt{\mu_i^2\sigma_i^2+\delta}}\}}{\min\{\underset{i,j\in[p],i\neq j}{\min}\frac{(\mu_i-\mu_j)(\sigma_i-\sigma_j)}{\sqrt{\mu_i^2\sigma_i^2+\delta}+\sqrt{\mu_j^2\sigma_j^2+\delta}},\underset{i\in[p]}{\min}\frac{\mu_i(\sigma_i-\sigma_{p+1})}{\sqrt{\mu_i^2\sigma_i^2+\delta}}\}}.
    \end{aligned}
  \end{equation*}
  Moreover, the new metric~\cref{eq: new metric in SVD} indeed improves the condition number of the Riemannian Hessian, i.e., 
  \[\kappa_\mathrm{new}(\Hess_\mathrm{new}\!f(\matu^*,\matv^*))\leq\kappa_\eucmetric(\Hess_\eucmetric\!f(\matu^*,\matv^*)).\]
\end{proposition}

\subsection{RGD and RCG for truncated singular value decomposition}
Let $\tensM$ be endowed with the Riemannian metric~\cref{eq: new metric in SVD}. We apply the Riemannian gradient descent (\cref{alg: RGD}) and Riemannian conjugate gradient (\cref{alg: RCG}) methods to solve the SVD problem~\cref{eq: SVD} in~\cref{alg: RGD SVD,alg: RCG SVD}. Note that the retraction mapping is based on the QR factorization, i.e., $\retr_{(\matu,\matv)}(\eta):=(\qf(\matu+\eta_1),\qf(\matv+\eta_2))$ for $\eta\in\tangent_{(\matu,\matv)}\!\tensM$.
The vector transport in~\cref{alg: RCG SVD} is defined by the projection operator~\cref{eq: projection operator in SVD}.
\begin{algorithm}[htbp]
  \caption{RGD for SVD}\label{alg: RGD SVD}
  \begin{algorithmic}[1]
      \REQUIRE $\tensM$ endowed with a metric~\cref{eq: new metric in SVD}, initial guess $(\matu^{(0)},\matv^{(0)})\in\tensM$, $t=0$. 
      \WHILE{the stopping criteria are not satisfied}
          \STATE Compute $\eta^{(t)}=-\grad_g f(\matu^{(t)},\matu^{(t)})$ by~\cref{eq: SVD gradient g}.
          \STATE Compute the stepsize $s^{(t)}$ by Armijo backtracking line search~\cref{eq: Armijo backtracking line search}. 
          \STATE Update $\matu^{(t+1)}=\qf(\matu^{(t)}+s^{(t)}\eta_1^{(t)})$, $\matv^{(t+1)}=\qf(\matv^{(t)}+s^{(t)}\eta_2^{(t)})$; $t= t+1$.
      \ENDWHILE
      \ENSURE $(\matu^{(t)},\matv^{(t)})\in\tensM$. 
  \end{algorithmic}
\end{algorithm}
\begin{algorithm}[htbp]
  \caption{RCG for SVD}\label{alg: RCG SVD}
  \begin{algorithmic}[1]
      \REQUIRE $\tensM$ endowed with a metric~\cref{eq: new metric in SVD}, initial guess $(\matu^{(0)},\matv^{(0)})\in\tensM$, $t=0$, $\beta^{(0)}=0$. 
      \WHILE{the stopping criteria are not satisfied}
          \STATE Compute $\eta^{(t)}=-\grad_g f(\matu^{(t)},\matu^{(t)})+\beta^{(t)}\Pi_{g,(\matu^{(t)},\matv^{(t)})}(\eta^{(t-1)})$ by~\cref{eq: SVD gradient g}. 
          \STATE Compute the stepsize $s^{(t)}$ by Armijo backtracking line search~\cref{eq: Armijo backtracking line search}. 
          \STATE Update $\matu^{(t+1)}=\qf(\matu^{(t)}+s^{(t)}\eta_1^{(t)})$, $\matv^{(t+1)}=\qf(\matv^{(t)}+s^{(t)}\eta_2^{(t)})$; $t= t+1$.
      \ENDWHILE
      \ENSURE $(\matu^{(t)},\matv^{(t)})\in\tensM$. 
  \end{algorithmic}
\end{algorithm}

\subsection{Numerical validation}
We compare the performance of~\cref{alg: RGD SVD,alg: RCG SVD} with RGD and RCG under the Euclidean metric in~\cite{sato2013riemannian}. The proposed preconditioned metric~\cref{eq: new metric in SVD}, which has a right preconditioning effect, is denoted by ``(R12)''. 
We set $m=1000$, $n=500$, $p=10$, and $\matN=\diag(p,p-1,\dots,1)$. The matrix $\mata$ is constructed by $\mata=\matu^*\Sigma(\matv^*)^\T$, where the entries of $\matu^*\in\mathbb{R}^{m\times p}$ and $\matv^*\in\mathbb{R}^{n\times p}$ are firstly sampled i.i.d. from the uniform distribution on $[0,1]$, and $\matu^*$ and $\matv^*$ are orthogonalized by QR factorization. We set $\Sigma:=\diag(1,\gamma,\gamma^2,\dots,\gamma^{p-1})$ and $\gamma=1/1.5$. The implementation of RGD and RCG is the same as~\cref{sec: CCA}.

Numerical results are shown in~\cref{fig: SVD results,fig: SVD computational cost,tab: SVD results}. We have similar observations as the previous experiments in~\cref{sec: CCA}. First, the proposed methods significantly outperform RGD(E) and RCG(E) with fewer iterations since the proposed metric benefits from the second-order information. Second, computational cost per iteration of~\cref{alg: RGD SVD,alg: RCG SVD} is comparable to RGD(E) and RCG(E) respectively. Third, \cref{tab: SVD results} shows that the subspace distances are smaller than $10^{-6}$ in RGD(R12) and RCG(R12), which indicates that the sequences generated by proposed methods converge to the correct subspace. 

\begin{figure}[htbp]
  \centering
  \subfigure{\includegraphics[width=0.48\textwidth]{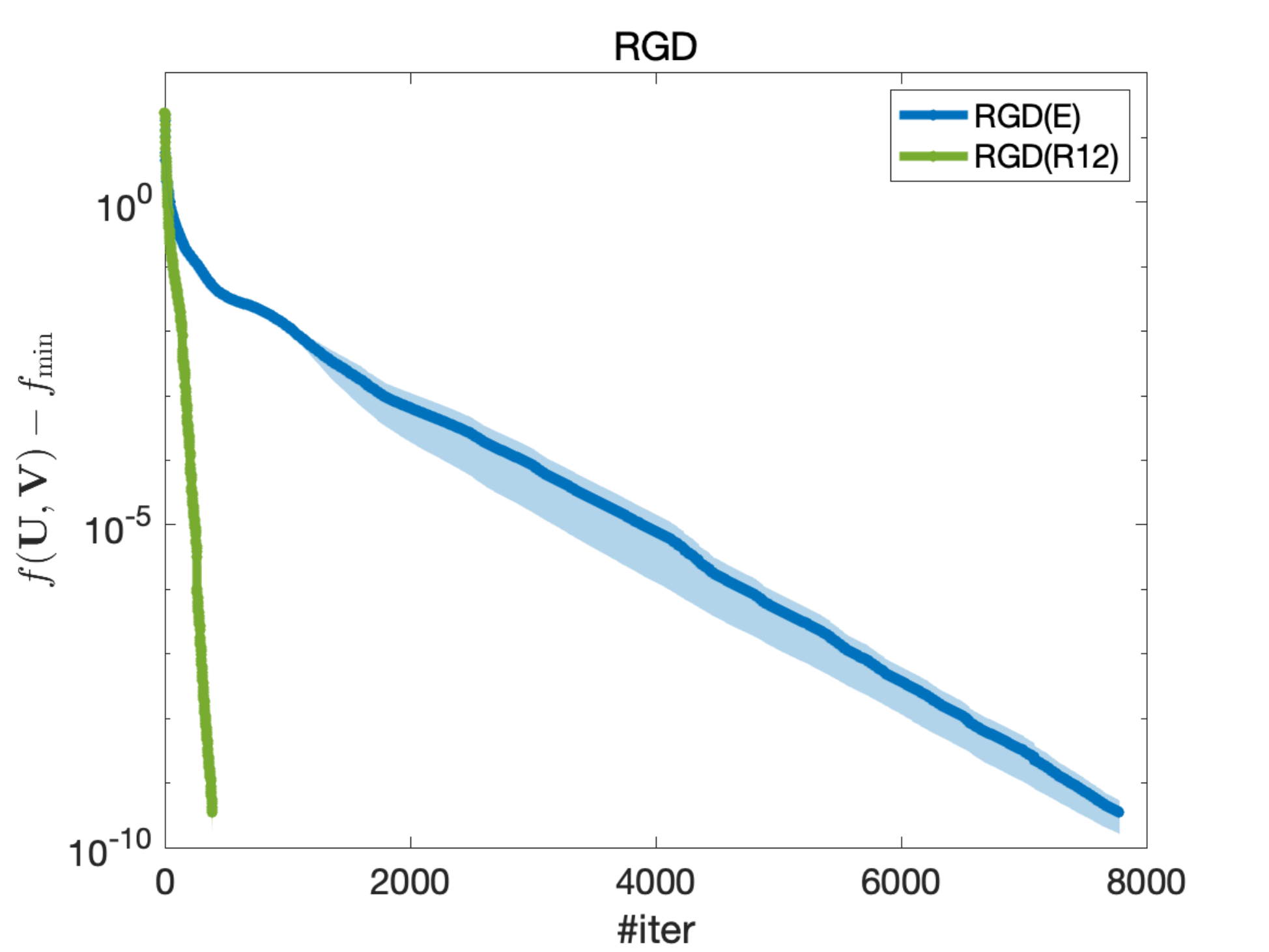}}
  \subfigure{\includegraphics[width=0.48\textwidth]{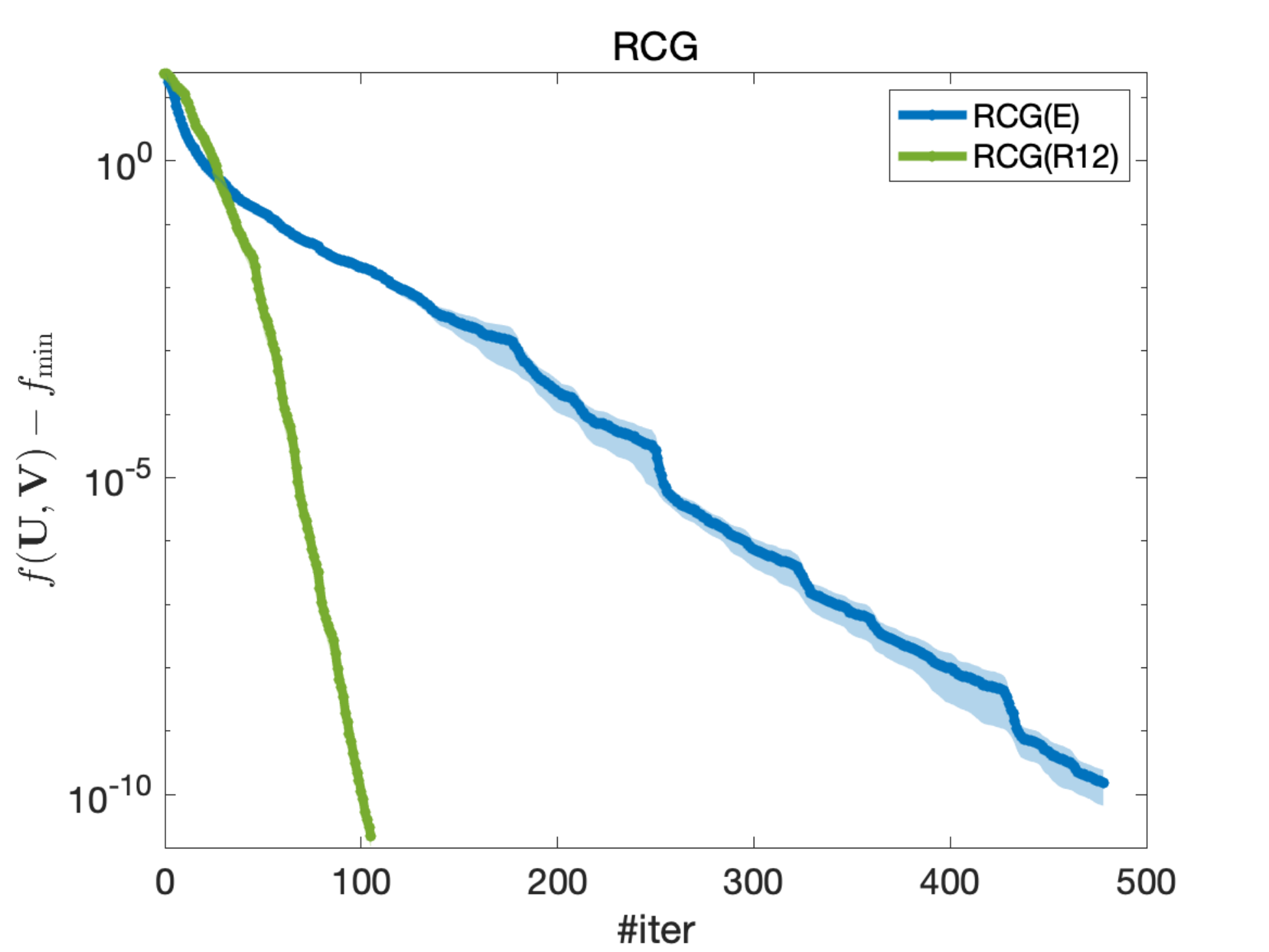}}
  \vspace{-1mm}
  \caption{Numerical results for the SVD problem for $m=1000$, $n=500$, and $p=10$. Left: RGD. Right: RCG. Each method is tested for 10 runs}
  \label{fig: SVD results}
\end{figure}

\begin{figure}[htbp]
  \centering
  \includegraphics[width=0.9\textwidth]{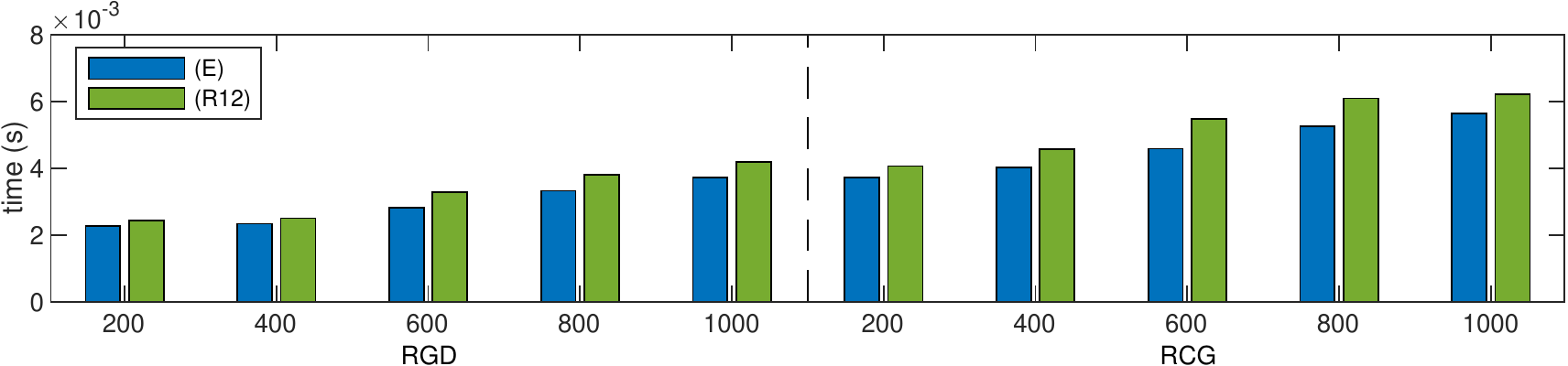}
  \vspace{-1mm}
  \caption{Average computation time per iteration for RGD (left) and RCG (right) under the Euclidean and proposed metric for $m=1000$, $p=10$, and $n=200,400,\dots,1000$}
  \label{fig: SVD computational cost}
\end{figure}

In addition, we compute the condition numbers of $\Hess\!f(\matu^*,\matv^*)$ under two metrics. It follows from the construction of $\mata$ and~\cref{prop: condition number of SVD} that 
\begin{equation*}
  \begin{aligned}
    \kappa(\Hess_{\eucmetric}\!f(\matu^*,\matv^*))&=\frac{(\mu_1+\mu_2)(\gamma+1)}{(\mu_{p-1}-\mu_p)(\gamma^{p-2}-\gamma^{p-1})}=\frac{153389}{63}\approx 2.43\times 10^3,\\
    \kappa(\Hess_{\mathrm{new}}\!f(\matu^*,\matv^*))&=\frac{(\mu_1+\mu_2)(1+\gamma)}{(\mu_1-\mu_2)(1-\gamma)}=95,
  \end{aligned}
\end{equation*}
which exactly coincide with the numerical results in~\cref{tab: SVD results}. Therefore, the lower condition number suggests faster convergence of the proposed methods.

\begin{table}[htbp]
  \centering
  \footnotesize
  \caption{Convergence results of the SVD problem for $m=1000$, $n=500$, and $p=10$}
  \label{tab: SVD results}
  \begin{tabular}{lcrrccccc}
    \toprule
    metric & \multirow{1}*{method} & \multirow{1}*{\#iter} & \multirow{1}*{time (s)} 
    & \multirow{1}*{gnorm} & \multirow{1}*{$D(\matu,\matu^*)$} & \multirow{1}*{$D(\matv,\matv^*)$} & $\kappa_g$\\
    \midrule
    \multirow{2}*{(E)} & RGD & 7781 & 117.29 & 9.64e-07 & 4.53e-05 & 4.53e-05 & \multirow{2}*{2.43e+03}\\
     & RCG & 478 & 5.44 & 8.54e-07 & 2.00e-05 & 2.00e-05 & \\
    \multirow{2}*{(R12)} & RGD & 387 & 3.41 & 8.72e-07 & \textbf{2.38e-15} & \textbf{1.38e-15} & \multirow{2}*{\bf{9.50e+01}} \\
     & \textbf{RCG} & \textbf{105} & \textbf{1.45} & \textbf{7.88e-07} & 3.26e-07 & 3.83e-07 & \\
    \bottomrule
\end{tabular}
\end{table}

\section{Application to matrix and tensor completion}\label{sec: LRMC and LRTC}
In this section, we investigate the matrix and tensor completion problem. Given a partially observed tensor $\tensA\in\mathbb{R}^{n_1\times n_2\times\cdots\times n_d}$ on an index set $\Omega\subseteq[n_1]\times[n_2]\times\cdots\times[n_d]$, the goal of tensor completion is to recover the tensor $\tensA$ from its entries on $\Omega$. Note that it boils down to matrix completion for $d=2$.

There are several different formulations in tensor completion. One type is based on the nuclear norm minimization, e.g.,~\cite{candes2012exact,liu2012tensor}. These methods require working with full-size tensors. Instead, tensor decompositions---which take advantage of the low-rank structure of a tensor---reduce the number of parameters in search space. Therefore, it is economical to formulate the tensor completion problem based on a tensor decomposition, which leads to an optimization problem on a product manifold 
\begin{equation}\label{eq:original problem} 
  \min f(x):=\frac{1}{2p}\left\| \proj_\Omega(\tau(x)-\tensA)\right\|_\mathrm{F}^2,\ 
  \subjectto\ x\in\tensM=\tensM_1\times\tensM_2\times\cdots\times\tensM_K,
\end{equation}
where $p:=|\Omega|/(n_1n_2\cdots n_d)$ is the sampling rate, $\proj_\Omega$ refers to the projection operator onto $\Omega$, i.e, $\proj_\Omega(\tensX)(i_1,\dots,i_d)=\tensX(i_1,\dots,i_d)$ if $(i_1,\dots,i_d)\in\Omega$, otherwise $\proj_\Omega(\tensX)(i_1,\dots,i_d)=0$ for $\tensX\in\mathbb{R}^{n_1\times\cdots\times n_d}$, and $\tau(x)$ denotes the tensor decomposition with components $x_k\in\tensM_k$ for $k\in[K]$ and $x=(x_1,x_2,\dots,x_K)$.

Since computing the Euclidean Hessian $\nabla^2 f(x)$ can be complicated, Kasai and Mishra~\cite{kasai2016low} introduced a preconditioned metric based on the block diagonal approximation of $\nabla^2 f(x)$ for tensor completion in Tucker decomposition. More recently, the idea became prosperous in low-rank tensor approximation and completion for other tensor formats, e.g.,~\cite{breiding2018riemannian,dong2022new,cai2022tensor,gao2024riemannian}, see~\cref{tab: covered existing works} for details. 
In summary, the metric was developed by constructing an operator $\bartensH(x)$ based on the diagonal blocks of the Hessian of the cost function $\phi(x):=\frac12\|\tau(x)-\tensA\|_\mathrm{F}^2$, i.e., 
\[\bartensH(x)[\eta]:=(\partial_{11}^2 \phi(x)[\eta_1],\dots,\partial_{KK}^2 \phi(x)[\eta_K])\quad\text{for } \eta=(\eta_1,\eta_2,\dots,\eta_K)\in\tangent_x\!\tensM,\]
where $\partial_{kk}^2 \phi(x)[\eta_k]:=\lim_{h\to 0}(\partial_k\phi(x_1,\dots,x_{k-1},x_k+h\eta_k,x_{k+1},\dots,x_K)-\partial_k\phi(x))/h$ for $k\in[K]$. Note that this preconditioning approach coincides with the exact block diagonal preconditioning in~\cref{subsec: diagonal approximation}. Alternatively, observing that the cost function $f$ in~\cref{eq:original problem} enjoys a least-squares structure, we can also adopt the Gauss--Newton type preconditioning in~\cref{subsec: Gauss--Newton} to solve~\cref{eq:original problem}.

\subsection{Gauss--Newton method for tensor ring completion}
Since tensor ring decomposition has been shown effective for the tensor completion problem, e.g.,~\cite{gao2024riemannian}, we consider the following tensor ring completion problem
\begin{equation}\label{eq: original TR problem} 
	\begin{array}{cc}
			\min\limits_{\tensU_k\in\mathbb{R}^{r_{k-1}\times n_k\times r_k}}&\ f(\tensU_1,\dots,\tensU_d):=\frac{1}{2p}\left\| \proj_\Omega(\llbracket\tensU_1,\dots,\tensU_d\rrbracket)-\proj_\Omega(\tensA)\right\|_\mathrm{F}^2, 
	\end{array}
\end{equation}
where $\llbracket\tensU_1,\dots,\tensU_d\rrbracket$ denotes the tensor ring decomposition~\cite{zhao2016tensor}. Specifically, given $\tensX=\llbracket\tensU_1,\dots,\tensU_d\rrbracket$ with $\tensU_k\in\mathbb{R}^{r_{k-1}\times n_k\times r_k}$ for $k\in[d]$ and $r_0=r_d$, the $(i_1,i_2,\dots,i_d)$-th element of $\tensX$ is defined by 
\[\tensX(i_1,i_2,\dots,i_d):=\tr(\matU_1(i_1)\matU_2(i_2)\cdots\matU_d(i_d)),\]
where $\matU_k(i_k):=\tensU_k(:,i_k,:)\in\mathbb{R}^{r_{k-1}\times r_k}$ refers the $i_k$-th lateral slice of the tensor $\tensU$ for $i_k\in[n_k]$. Since the $k$-th unfolding matrix of $\tensX$ satisfies $\matx_{(k)}=(\tensU_{k})_{(2)}(\tensU_{\neq k})_{(2)}$, problem~\cref{eq: original TR problem} can be reformulated by introducing~\cite{gao2024riemannian} $\matW_k:=(\tensU_{k})_{(2)}$ and $\matW_{\neq_k}:=(\tensU_{\neq k})_{(2)}$, where $(\tensU_{k})_{(2)}$ and $(\tensU_{\neq k})_{(2)}$ are the 2-nd unfolding matrix of the tensor $\tensU_{k}$ and $\tensU_{\neq k}$ respectively, and $\tensU_{\neq k}\in\mathbb{R}^{r_{k-1}\times\prod_{j\neq k}n_j\times r_{k}}$ is defined by its lateral slice matrices, i.e., $\matU_{\neq k}(1+\sum_{\ell \neq k, \ell = 1}^d(i_\ell-1)J_\ell){:=}(\prod_{j=k+1}^{d}\matU_{j}(i_{j})\prod_{j=1}^{k-1}\matU_{j}(i_{j}))^\T$ with $J_\ell := \prod_{m = 1, m \neq k}^{\ell-1} n_m$. Consequently, a reformulation of~\cref{eq: original TR problem} is given by
\begin{equation}
  \label{eq: tensor ring completion}
  \begin{aligned}
    \min_{\vec{\matW}}\ &\ \ f(\vecmatW):=\frac{1}{2p}\|\proj_{\Omega} (\tau(\vecmatW)-\tensA) \|_\mathrm{F}^2\\
    \subjectto\ &\ \ \vec{\matW}\in\tensM=\mathbb{R}^{n_1\times r_0r_{1}}\times\mathbb{R}^{n_2\times r_1r_{2}}\times\cdots\times\mathbb{R}^{n_d\times r_{d-1}r_{d}},
  \end{aligned}
\end{equation}
where the mapping $\tau$ is defined by $\tau:\vecmatW\mapsto\llbracket\ten_{(2)}(\matW_1),\ten_{(2)}(\matW_2),\dots,\ten_{(2)}(\matW_d)\rrbracket$ and $\ten_{(2)}(\cdot)$ is the second tensorization operator. 

Noticing that the cost function $f$ in~\cref{eq: tensor ring completion} enjoys a least-squares structure: $f(\vecmatW)=\frac12\|F(\vecmatW)\|_\mathrm{F}^2$, where $F(\vecmatW)=\proj_{\Omega} (\tau(\vecmatW)-\tensA)/\sqrt{p}$ is a smooth function, we adopt the Gauss--Newton type preconditioning to solve~\cref{eq: tensor ring completion}. Since the search space $\tensM$ is flat, the RGD method under the metric~\cref{eq: Gauss--Newton metric} is essentially a Euclidean Gauss--Newton method (see~\cite[\S 10.3]{nocedal2006numerical}). We list the Gauss--Newton method in~\cref{alg: RGN}; see~\cref{app: TR-GN} for implementation details. 
\begin{algorithm}[htbp]
  \caption{Gauss--Newton method for tensor ring completion (TR-GN)}\label{alg: RGN}
  \begin{algorithmic}[1]
      \REQUIRE $\tensM$ endowed with a metric $g$, initial guess $\vecmatW^{(0)}\in\tensM$, $t=0$. 
      \WHILE{the stopping criteria are not satisfied}
          \STATE Compute $\eta^{(t)}$ by solving \cref{eq: search direction of GN}. 
          \STATE Update $\vecmatW^{(t+1)}=\vecmatW^{(t)}+\eta^{(t)}$; $t=t+1$.
      \ENDWHILE
      \ENSURE $\vecmatW^{(t)}\in\tensM$. 
  \end{algorithmic}
\end{algorithm}

Since the tensor ring decomposition is complicated, we leave the condition number results analogous to~\cref{prop: condition number of Shustin,prop: condition number of CCA in new metric,prop: condition number of SVD} for future work. Nevertheless, if the sequence generated by the Gauss--Newton method (\cref{alg: RGN}) converges to $\vecmatW^*\in\tensM$ with $F(\vecmatW^*)=0$, the Gauss--Newton method enjoys superlinear convergence, see~\cite[\S 8.4.1]{absil2009optimization} for Riemannian Gauss--Newton method and~\cite[\S 10.3]{nocedal2006numerical} for the Euclidean Gauss--Newton method.

\subsection{Numerical validation}
We compare~\cref{alg: RGN} with the Riemannian gradient descent (TR-RGD) and the Riemannian conjugate gradient (TR-RCG) methods in~\cite{gao2024riemannian} under the metric $g_{\vecmatW}(\xi,\eta):=\sum_{k=1}^d\langle\xi_k,\eta_k(\matW_{\neq k}^\T\matW_{\neq k}^{}+\delta\matI_{r_{k-1}r_k})\rangle$ for $\xi,\eta\in\tangent_{\vecmatW}\!\tensM$, where $\delta>0$ is a constant. The codes for TR-RGD, TR-RCG, and TR-GN methods are available at \url{https://github.com/JimmyPeng1998/LRTCTR}.

The tensor $\tensA\in\mathbb{R}^{n_1\times n_2\times \cdots \times n_d}$ is constructed by $\tensA=\tau(\vecmatW^*)$ and each entry of $\vecmatW^*\in\tensM$ is uniformly sampled from $[0,1]$. The initial guess $\vecmatW^{(0)}\in\tensM$ is generated in a same fashion. Given the sampling rate $p$, we randomly select $pn_1n_2\cdots n_d$ samples from $[n_1]\times[n_2]\times\cdots\times[n_d]$ to formulate the sampling set $\Omega$. We set $d=3$, $n_1=n_2=n_3=100$, $p=0.05$, TR ranks $\vecr^*=(1,1,1),(2,2,2),\dots,(8,8,8)$, and $\delta=10^{-15}$.

We specify the default settings of all methods. 
The stepsize rule for TR-RGD method and the TR-RCG method is the Armijo backtracking line search~\cref{eq: Armijo backtracking line search}. The conjugate gradient parameter is set to be the Riemannian version~\cite{boumal2014manopt} of the modified Hestenes--Stiefel rule. The parameters in~\cref{eq: Armijo backtracking line search} are $\rho=0.3$, $a=2^{-13}$, and $s_0=1$. 
The performance of each method is evaluated by the training error $\varepsilon_{\Omega}(\vecmatW^{(t)}):=\|\proj_{\Omega}(\tau(\vecmatW^{(t)}))-\proj_{\Omega}(\tensA)\|_\mathrm{F}/\|\proj_{\Omega}(\tensA)\|_\mathrm{F}$ and the test error $\varepsilon_{\Gamma}(\vecmatW^{(t)})$, where $\Gamma$ is a test set different from $\Omega$ and we set $|\Gamma|=100$. A method is terminated if one of the following stopping criteria is achieved: 1) training error $\varepsilon_{\Omega}(\vecmatW^{(t)})<10^{-14}$; 2) the maximum iteration $1000$; 3) the relative change $\lvert{(\varepsilon_{\Omega}(\vecmatW^{(t)})-\varepsilon_{\Omega}(\vecmatW^{(t-1)}))}/{\varepsilon_{\Omega}(\vecmatW^{(t-1)})}\rvert<\varepsilon$; 4) the stepsize $s^{(t)}<10^{-10}$.

Numerical results are illustrated in~\cref{fig: numerical results for LRTC,fig: time required for LRTC}. On the one hand, we observe that the TR-GN method has faster convergence than TR-RGD and TR-RCG since TR-GN exploits more second-order information of $\nabla^2 f(\vecmatW)$, while the preconditioned metric in TR-RGD and TR-RCG only takes advantage of its diagonal blocks. On the other hand,~\cref{fig: time required for LRTC} suggests that the computation time for TR-GN to reach the stopping criteria grows faster than TR-RGD and TR-RCG as TR rank $\vecr^*$ increases. In other words, there is a trade-off between exploiting second-order information and the computational efficiency.

\begin{figure}[htbp]
  \centering
  \subfigure{\includegraphics[width=0.48\textwidth]{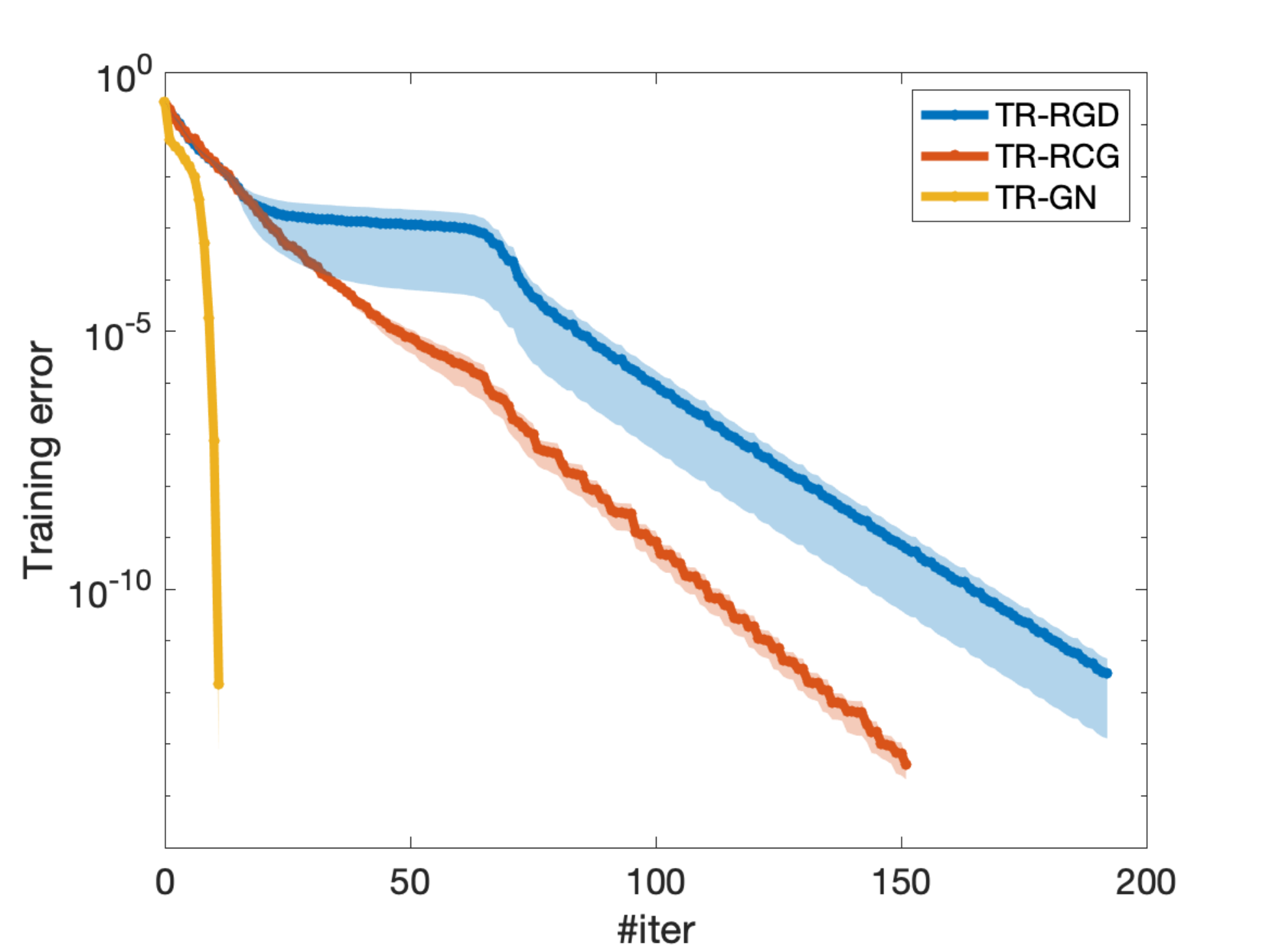}}
  \subfigure{\includegraphics[width=0.48\textwidth]{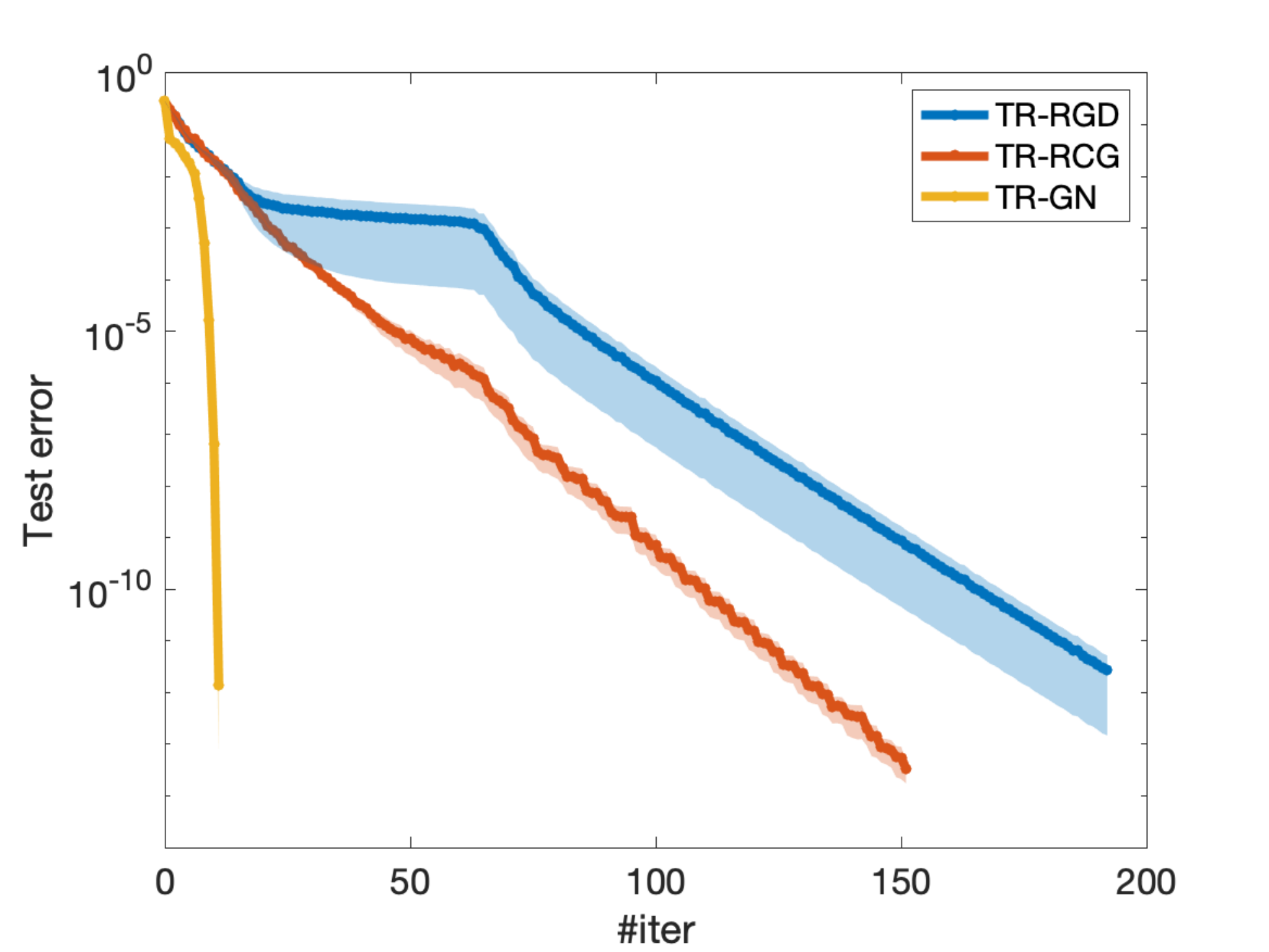}}
    \vspace{-1mm}
  \caption{Training and test errors for TR rank $\vecr^*=(5,5,5)$. Each method is tested for 10 runs}
  \label{fig: numerical results for LRTC}
\end{figure}

\begin{figure}[htbp]
  \centering
  \includegraphics[width=0.9\textwidth]{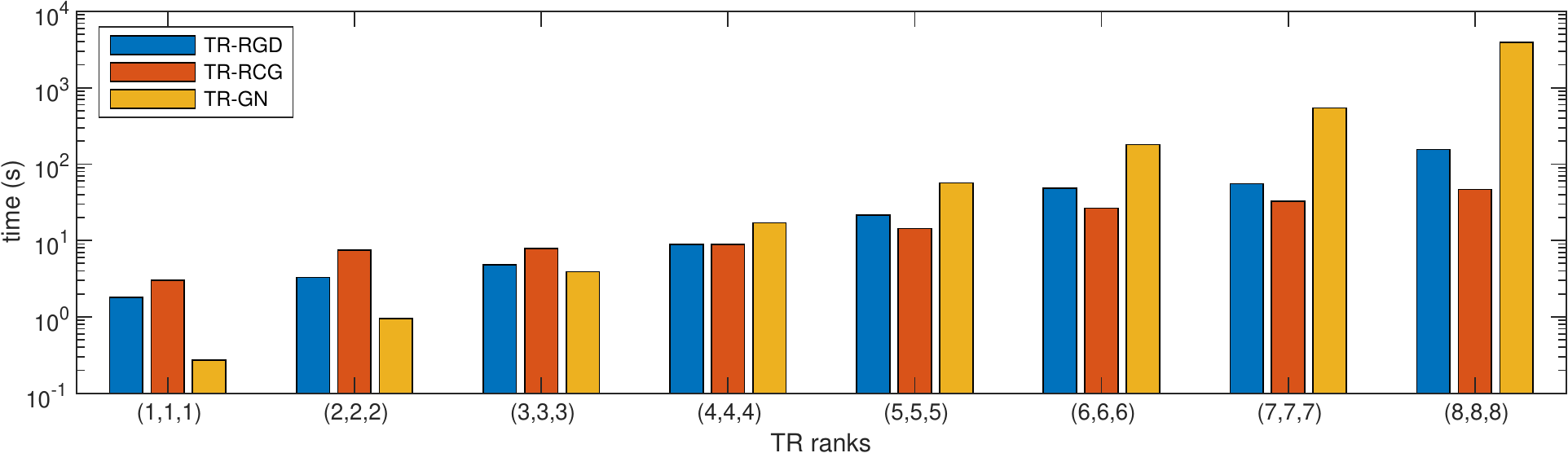}
    \vspace{-1mm}
  \caption{Computation time required for each method to reach the stopping criteria for TR rank $\vecr^*=(1,1,1),(2,2,2),\dots,(8,8,8)$}
  \label{fig: time required for LRTC}
\end{figure}

\section{Conclusions and future works} 
\label{sec: conclusions}
The performance of the Riemannian methods varies with different metrics. We have shown that an exquisitely constructed metric is indeed helpful to accelerate the Riemannian methods. Specifically, we have proposed a general framework for optimization on product manifolds endowed with a preconditioned metric and three specific approaches to construct an operator that aims to approximate the Riemannian Hessian. Conceptually, various existing methods including the Riemannian Gauss--Newton method and the block-Jacobi preconditioning in numerical linear algebra can be interpreted by the proposed framework with specific metrics. We have tailored novel preconditioned metrics based on the proposed framework for canonical correlation analysis and truncated singular value decomposition and have shown the effect of the proposed metric by computing the condition number of the Riemannian Hessian at the local minimizer, which indeed improves the condition numbers. Numerical results verify that a delicate metric does improve the performance of the Riemannian optimization methods. 

In the future, we intend to specify the proposed framework to other problems. Moreover, taking into account the block structure of product manifolds, parallel computing is capable of further accelerating the Riemannian optimization methods.

\section*{Acknowledgments}
We would like to thank the editor and two anonymous reviewers for insightful comments.

\appendix

\section{Computational details in~\cref{eg: eg1}} \label{app: eg 1.1}
In~\cref{eg: eg1}, we investigate a class of preconditioned metrics defined by 
\[g_{\lambda,\vecx}(\xi,\eta)=\langle\xi,\matb_\lambda\eta\rangle\quad\text{for }\xi,\eta\in\tangent_\vecx\!\tensM_\matb,\]
where $\lambda\in\mathbb{R}$ such that the matrix $\matb_\lambda:=\lambda\matI_n+(1-\lambda)\matB$ is positive definite. The Riemannian gradients at $x\in\tensM_\matb$ under these metrics are
  \begin{equation*}
    \grad_{g_{\lambda}} f(x)=-(\matI_n-\frac{\matb_\lambda^{-1}\matb\vecx\vecx^\T\matb}{\vecx^\T\matb\matb_\lambda^{-1}\matb\vecx})\matb_\lambda^{-1}\vecb=-\matb_\lambda^{-1}\vecb+\frac{\vecx^\T\matb\matb_\lambda^{-1}\vecb}{\vecx^\T\matb\matb_\lambda^{-1}\matb\vecx}\matb_\lambda^{-1}\matb\vecx,
  \end{equation*}
  by following~\cref{prop: orthogonal projection operator of new metric}. Subsequently, the update rule for RGD under the metric $g_{\lambda}$ is  $\vecx^{(t+1)}={\bar{\vecx}^{(t)}}/{\|\bar{\vecx}^{(t)}\|_\matB}$ with $\bar{\vecx}^{(t)}=\vecx^{(t)}-s^{(t)}\grad_{g_{\lambda}} f(\vecx^{(t)})$, where we adopt the polar retraction~\cite[(3.3)]{shustin2023riemannian}.
  The Riemannian Hessian of $f$ at $\vecx^*$ along $\eta\in\tangent_{\vecx^*}\!\tensM_\matb$ can be computed by
  \begin{equation*}
  	\begin{aligned}
  		\Hess_{g_{\lambda}}\!f(\vecx^*)[\eta]&=\Pi_{g_{\lambda},\vecx^*}(\mathrm{D}\grad_{g_{\lambda}} f(\vecx^*)[\eta])=\Pi_{g_{\lambda},\vecx^*}(\|\matb^{-1}\vecb\|_\matb{\matb_\lambda}^{-1}\matb\eta),
  	\end{aligned}
  \end{equation*}
  since $\grad_{g_{\lambda}} f(\vecx^*)=\Pi_{g_{\lambda},\vecx^*}(-\matb_\lambda^{-1}\vecb)=0$. The Rayleigh quotient~\cref{eq: Rayleigh quotient} is given by 
  \[q(\eta)=\|\matb^{-1}\vecb\|_\matb\cdot\frac{\langle\eta,\matb\eta\rangle}{\langle\eta,{\matb_\lambda}\eta\rangle}\quad \text{for }\eta\in\tangent_{\vecx^*}\!\tensM_\matb.\]
  Consequently, we can compute the condition number of $\Hess_{g_{\lambda}}\!f(\vecx^*)$ in a same fashion as~\cref{prop: condition number of CCA in new metric}. Note that if $\lambda=0$, the Rayleigh quotient boils down to a constant $\|\matb^{-1}\vecb\|_\matb$ and thus $\kappa_{g_{0}}(\Hess_{g_{0}}\!f(\vecx^*))=1$.

\section{Proof of~\cref{prop: diff metric}}\label{app: proof of diff metric}
\begin{proof}
  It suffices to prove the first inequality, and the other can be proved in a same fashion. Let $(\tensU,\varphi)$ be a chart of the manifold $\tensM$, and $E_i$ be the $i$-th coordinate vector field. For vector fields $\zeta=\sum_i\alpha_iE_i$ and $\chi=\sum_i\beta_iE_i$, it follows the definition of the Riemannian metric $g$ that $g_x(\zeta_x,\chi_x)=\sum_{i,j}g_{ij}\alpha_i\beta_j=\hat{\zeta}_{\hat{x}}^\T\matG_{\hat{x}}\hat{\chi}_{\hat{x}}$,
  where $\hat{x}:=\varphi(x)$, $\hat{\zeta}_{\hat{x}}:=\mathrm{D}\varphi(\varphi^{-1}(\hat{x}))[\zeta_x]$, $\hat{\chi}_{\hat{x}}:=\mathrm{D}\varphi(\varphi^{-1}(\hat{x}))[\chi_x]$, and the $(i,j)$-th element of $\matG_{\hat{x}}$ is $g_{ij}:=g(E_i,E_j)$. 
  Denote $\zeta_x:=\grad_{g} f(x)$ and $\chi_x:=\grad_{\tilde{g}} f(x)$. It follows from the coordinate expression~\cite[\S 3.6]{absil2009optimization} that $\hat{\zeta}_{\hat{x}}=\matG^{-1}_{\hat{x}}\nabla\hat{f}(\hat{x})$ and $\hat{\chi}_{\hat{x}}=\tilde{\matG}_{\hat{x}}^{-1}\nabla\hat{f}(\hat{x})$,
  where $\hat{f}(\hat{x}):=f\circ\varphi^{-1}(\hat{x})$ and $\nabla\hat{f}$ refers to the Euclidean gradient of $\hat{f}$. We obtain that
  \[
  g_{x}^{}(\grad_{g} f(x),\grad_{\tilde{g}} f(x))=\hat{\zeta}_{\hat{x}}^\T\matG_{\hat{x}}^{}\hat{\chi}_{\hat{x}}^{}=(\nabla\hat{f}(\hat{x}))^\T\tilde{\matG}_{\hat{x}}^{-1}\nabla\hat{f}(\hat{x})\geq 0.
  \]
  The equality holds if and only if $\nabla\hat{f}(\hat{x})=0$, i.e., $\grad_{g} f(x)=\grad_{\tilde{g}} f(x)=0$. Moreover, if $\grad_g f(x)=0$, it follows that $\hat{\zeta}_{\hat{x}}=0$, and hence $\hat{\chi}_{\hat{x}}=\tilde{\matG}_{\hat{x}}^{-1}\nabla\hat{f}(\hat{x})=\tilde{\matG}_{\hat{x}}^{-1}\matG_{\hat{x}}^{}\matG^{-1}_{\hat{x}}\nabla\hat{f}(\hat{x})=\tilde{\matG}_{\hat{x}}^{-1}\matG_{\hat{x}}^{}\hat{\zeta}_{\hat{x}}^{}=0$, i.e., $\grad_{\tilde{g}} f(x)=0$.
\end{proof}

\section{Proof of~\cref{prop: condition number of Shustin}}\label{app: proof}
\cref{prop: condition number of Shustin} gives the condition number of the Riemannian Hessian of $f$ at $(\matu^*,\matv^*)$ under the metric~\cref{eq: CCA metric by Shustin}. A concrete proof of~\cref{prop: condition number of Shustin} is given as follows.
  \begin{proof} 
    Since $(\matu^*,\matv^*)$ is a critical point of $f$, it follows from $(\matu^*)^\T\Sigma_{xy}\matv^*=\Sigma$ and $\grad_g f(\matu^*,\matv^*)=0$ that
    \begin{equation}\label{eq: CCA stationarity}
      \Sigma_{xx}^{-1}\Sigma_{xy}^{}\matv^*=\matu^*\Sigma\quad\text{and}\quad\Sigma_{yy}^{-1}\Sigma_{xy}^\T\matu^*=\matv^*\Sigma,
    \end{equation}
    where $\Sigma=\diag(\sigma_1,\sigma_2,\dots,\sigma_m)$ and $\sigma_1>\sigma_2>\cdots>\sigma_m$ are the $m$ largest singular values of $\Sigma_{xx}^{-1/2}\Sigma_{xy}\Sigma_{yy}^{-1/2}$.
    Therefore, it follows from~\cref{eq: tangent space of gen St,eq: Riemannian Hessian CCA Shustin,eq: CCA stationarity} that
    \begin{equation*}
      g_{(\matu^*,\matv^*)}(\eta,\Hess_g\!f(\matu^*,\matv^*)[\eta])=\langle\eta_1,\Sigma_{xx}\eta_1\Sigma\matN\rangle-2\langle\eta_1,\Sigma_{xy}\eta_2\matN\rangle+\langle\eta_2,\Sigma_{yy}\eta_2\Sigma\matN\rangle
    \end{equation*}
    for $\eta=(\eta_1,\eta_2)\in\tangent_{(\matu^*,\matv^*)}\!\tensM$.

    The goal of the proof is to compute the condition number of $\Hess_g\!f(\matu^*,\matv^*)$. To this end, we calculate the Rayleigh quotient of $\Hess_g\!f(\matu^*,\matv^*)$ by~\cref{eq: Rayleigh quotient} and evaluate its upper and lower bounds. First, the Rayleigh quotient reads
      \begin{align}
        q(\eta)
        &=\frac{\langle\eta_1,\Sigma_{xx}\eta_1\Sigma\matN\rangle-2\langle\eta_1,\Sigma_{xy}\eta_2\matN\rangle+\langle\eta_2,\Sigma_{yy}\eta_2\Sigma\matN\rangle}{\langle\eta_1,\Sigma_{xx}\eta_1\rangle+\langle\eta_2,\Sigma_{yy}\eta_2\rangle}\nonumber\\
        &=\frac{\langle\tilde{\eta}_1,\tilde{\eta}_1\Sigma\matN\rangle-2\langle\tilde{\eta}_1,\Sigma_{xx}^{-1/2}\Sigma_{xy}\Sigma_{yy}^{-1/2}\tilde{\eta}_2\matN\rangle+\langle\tilde{\eta}_2,\tilde{\eta}_2\Sigma\matN\rangle}{\langle\tilde{\eta}_1,\tilde{\eta}_1\rangle+\langle\tilde{\eta}_2,\tilde{\eta}_2\rangle}\nonumber\\
        &=\frac{
        \begin{bmatrix}
            \rmvec(\tilde{\eta}_1)^\T & \rmvec(\tilde{\eta}_2)^\T
        \end{bmatrix}
        \begin{bmatrix}
            \Sigma\matN\otimes\matI_{d_x} & -\matN\otimes\matM\\
            -\matN\otimes\matM^\T & \Sigma\matN\otimes\matI_{d_y}
        \end{bmatrix}
        \begin{bmatrix}
            \rmvec(\tilde{\eta}_1)\\
            \rmvec(\tilde{\eta}_2)
        \end{bmatrix}
        }{\langle\tilde{\eta}_1,\tilde{\eta}_1\rangle+\langle\tilde{\eta}_2,\tilde{\eta}_2\rangle}\nonumber\\
        &=\frac{\sum\limits_{i=1}^m\mu_i
        \begin{bmatrix}
            (\tilde{\eta}_1(:,i))^\T & (\tilde{\eta}_2(:,i))^\T
        \end{bmatrix}
        \begin{bmatrix}
            \sigma_i\matI_{d_x} & -\matM\\
            -\matM^\T & \sigma_i\matI_{d_y}
        \end{bmatrix}
        \begin{bmatrix}
            \tilde{\eta}_1(:,i)\\
            \tilde{\eta}_2(:,i)
        \end{bmatrix}
        }{\langle\tilde{\eta}_1,\tilde{\eta}_1\rangle+\langle\tilde{\eta}_2,\tilde{\eta}_2\rangle},\label{eq: Rayleigh quotient of CCA}
      \end{align}
    where $\matM=\Sigma_{xx}^{-1/2}\Sigma_{xy}\Sigma_{yy}^{-1/2}$, $\tilde{\eta}_1=\Sigma_{xx}^{1/2}\eta_1$ and $\tilde{\eta}_2=\Sigma_{yy}^{1/2}\eta_2$. By using~\cref{eq: tangent space of gen St},  we can represent $\tilde{\eta}$ by 
    \begin{align}
        \tilde{\eta}&=(\Sigma_{xx}^{\frac12}\eta_1,\Sigma_{yy}^{\frac12}\eta_2)=(\Sigma_{xx}^{\frac12}\matu^*\matOmega_1+\Sigma_{xx}^{\frac12}\matu_{\Sigma_{xx}\perp}^*\matK_1,\Sigma_{yy}^{\frac12}\matv^*\matOmega_2+\Sigma_{yy}^{\frac12}\matv_{\Sigma_{yy}\perp}^*\matK_2)\nonumber\\
        &=(\bar{\matu}\matOmega_1+\bar{\matu}_\perp\matK_1,\bar{\matv}\matOmega_2+\bar{\matv}_\perp\matK_2),\label{eq: tilde eta}
    \end{align} 
    where $\bar{\matu}=\Sigma_{xx}^{\frac12}\matu^*\in\St(m,d_x)$, $\bar{\matv}=\Sigma_{yy}^{\frac12}\matv^*\in\St(m,d_y)$, $\bar{\matu}_\perp=\matu_{\Sigma_{xx}\perp}^*\in\St(d_x-m,d_x)$ and $\bar{\matv}_\perp=\Sigma_{yy}^{\frac12}\matv_{\Sigma_{yy}\perp}^*\in\St(d_y-m,d_y)$ satisfy $\bar{\matu}^\T\bar{\matu}_\perp=0$, $\bar{\matv}^\T\bar{\matv}_\perp=0$ and $\matM=[\bar{\matu}\ \bar{\matu}_\perp]\diag(\sigma_1,\dots,\sigma_r,0,\dots,0)[\bar{\matv}\ \bar{\matv}_\perp]^\T$ with $r=\rank(\matM)$ by~\cref{eq: cca solution}. Taking~\cref{eq: tilde eta} into~\cref{eq: Rayleigh quotient of CCA}, we obtain that
    \begin{align}
            q(\eta) & =\dfrac{\sum\limits_{i=1}^m\mu_i
            \begin{bmatrix}
              [\bar{\matu}\ \bar{\matu}_\perp]\bar{\matOmega}_1(:,i)\\ 
              [\bar{\matv}\ \bar{\matv}_\perp]\bar{\matOmega}_2(:,i)
            \end{bmatrix}^\T
            \begin{bmatrix}
                \sigma_i\matI_{d_x} \ -\matM\\
                -\matM^\T \ \sigma_i\matI_{d_y}
            \end{bmatrix}
            \begin{bmatrix}
              [\bar{\matu}\ \bar{\matu}_\perp]\bar{\matOmega}_1(:,i)\\ 
              [\bar{\matv}\ \bar{\matv}_\perp]\bar{\matOmega}_2(:,i)
            \end{bmatrix}}
            {\langle\tilde{\eta}_1,\tilde{\eta}_1\rangle+\langle\tilde{\eta}_2,\tilde{\eta}_2\rangle}\nonumber \\
            & =\dfrac{\sum\limits_{i=1}^m\mu_i\Big(-\sum\limits_{j=1}^r 2\sigma_j\bar{\matOmega}_1(j,i)\bar{\matOmega}_2(j,i)+\sum\limits_{j=1}^{d_x}\sigma_i\bar{\matOmega}_1(j,i)^2+\sum\limits_{j=1}^{d_y}\sigma_i\bar{\matOmega}_2(j,i)^2\Big)}{\|\bar{\matOmega}_1\|_\frob^2+\|\bar{\matOmega}_2\|_\frob^2},  \label{eq: Step 1} 
    \end{align}
    where $\bar{\matOmega}_\ell:=\begin{bmatrix}
      \matOmega_\ell\\
      \matK_\ell
    \end{bmatrix}$ for $\ell=1,2$.

    Subsequently, by using $\bar{\matOmega}_\ell(j,i)^2=\bar{\matOmega}_\ell(i,j)^2$ for $\ell=1,2$, $i,j\in[m]$, we regroup the terms $\bar{\matOmega}_\ell(j,i)^2$ in~\cref{eq: Step 1} and yield
    \begin{equation}\label{eq: upper bound}
        \begin{aligned}
            q(\eta) & \leq\dfrac{\sum\limits_{i=1}^m\mu_i\Big(\sum\limits_{j=1}^r\sigma_j(\bar{\matOmega}_1(j,i)^2+\bar{\matOmega}_2(j,i)^2)+\sum\limits_{j=1}^{d_x}\sigma_i\bar{\matOmega}_1(j,i)^2+\sum\limits_{j=1}^{d_y}\sigma_i\bar{\matOmega}_2(j,i)^2\Big)}{\|\bar{\matOmega}_1\|_\frob^2+\|\bar{\matOmega}_2\|_\frob^2}\\
            &=\dfrac{\sum\limits_{i=1}^m\Big(\sum\limits_{j=1}^r\bar{s}_{ij}(\bar{\matOmega}_1(j,i)^2+\bar{\matOmega}_2(j,i)^2)+\sum\limits_{j=r+1}^{d_x}\mu_i\sigma_i\bar{\matOmega}_1(j,i)^2+\sum\limits_{j=r+1}^{d_y}\mu_i\sigma_i\bar{\matOmega}_2(j,i)^2\Big)}{\|\bar{\matOmega}_1\|_\frob^2+\|\bar{\matOmega}_2\|_\frob^2}\\
            &\leq\max\left\{(\mu_1+\mu_2)(\sigma_1+\sigma_2)/2,\mu_1(\sigma_1+\sigma_{m+1})\right\},
        \end{aligned}
    \end{equation}
    where $\bar{s}_{ij}:=\begin{cases}
      (\mu_i+\mu_j)(\sigma_i+\sigma_j)/2,\ j=1,2,\dots,m;\\
      \mu_i(\sigma_i+\sigma_j),\ j=m+1,m+2,\dots,r
    \end{cases}$ for $i=1,2,\dots,m$. The equality holds if and only if: 1) $\bar{\matOmega}_1(j,i)=-\bar{\matOmega}_2(j,i)$ for all $i\in[m],j\in[r]$; 2) $\bar{\matOmega}_1(j,i)^2=0$ for all $i\in[m],j=r+1,r+2,\dots,d_x$; 3) $\bar{\matOmega}_2(j,i)^2=0$ for all $i\in[m],j=r+1,r+2,\dots,d_y$; 4) $\bar{\matOmega}_1(j,i)=\bar{\matOmega}_2(j,i)=0$ for $(i,j)\neq(i^*,j^*)$, where $(i^*,j^*)\in\argmax_{i\in[m],j\in[r],i\neq j}\bar{s}_{ij}\subseteq\{(1,2),(2,1),(1,m+1)\}$.

    Additionally, we compute the lower bound of Rayleigh quotient in a same fashion as~\cref{eq: upper bound} and yield
    \begin{equation}\label{eq: lower bound}
        \begin{aligned}
            q(\eta) & \geq\dfrac{\sum\limits_{i=1}^m\mu_i\Big(-\sum\limits_{j=1}^r\sigma_j(\bar{\matOmega}_1(j,i)^2+\bar{\matOmega}_2(j,i)^2)+\sum\limits_{j=1}^{d_x}\sigma_i\bar{\matOmega}_1(j,i)^2+\sum\limits_{j=1}^{d_y}\sigma_i\bar{\matOmega}_2(j,i)^2\Big)}{\|\bar{\matOmega}_1\|_\frob^2+\|\bar{\matOmega}_2\|_\frob^2}\\
            &=\dfrac{\sum\limits_{i=1}^m\Big(\sum\limits_{j=1}^r\underline{s}_{ij}(\bar{\matOmega}_1(j,i)^2+\bar{\matOmega}_2(j,i)^2)+\sum\limits_{j=r+1}^{d_x}\mu_i\sigma_i\bar{\matOmega}_1(j,i)^2+\sum\limits_{j=r+1}^{d_y}\mu_i\sigma_i\bar{\matOmega}_2(j,i)^2\Big)}{\|\bar{\matOmega}_1\|_\frob^2+\|\bar{\matOmega}_2\|_\frob^2}\\
             & \geq\min\{\min_{i,j\in[m],i\neq j}(\mu_i-\mu_j)(\sigma_i-\sigma_j)/2,\mu_m(\sigma_m-\sigma_{m+1})\},
        \end{aligned}
    \end{equation}
    where $\underline{s}_{ij}:=\begin{cases}
      (\mu_i-\mu_j)(\sigma_i-\sigma_j)/2,\ j=1,2,\dots,m;\\
      \mu_i(\sigma_i-\sigma_j),\ j=m+1,m+2,\dots,r
    \end{cases}$ for $i=1,2,\dots,m$. The equality holds if and only if: 1) $\bar{\matOmega}_1(j,i)=\bar{\matOmega}_2(j,i)$ for all $i\in[m],j\in[r]$; 2) $\bar{\matOmega}_1(j,i)^2=0$ for all $i\in[m],j=r+1,\dots,d_x$; 3) $\bar{\matOmega}_2(j,i)^2=0$ for all $i\in[m],j=r+1,\dots,d_y$; 4) $\bar{\matOmega}_1(j,i)=\bar{\matOmega}_2(j,i)=0$ for $(i,j)\neq(i^*,j^*)$, where $(i^*,j^*)\in\argmin_{i\in[m],j\in[r],i\neq j}\underline{s}_{ij}$.
    
    Since the inequalities in~\cref{eq: upper bound,eq: lower bound} are tight, the proof is completed.
    \end{proof}

\section{Proof of~\cref{prop: orthogonal projection operator of new metric}}\label{app: proof of orth proj}
\begin{proof}
It suffices to prove $\Pi_{\mathrm{new},\matu}(\bareta_1)=\bareta_1-\matu\matS_1^{}\matM_{1,2}^{-1}$, and the others can be obtained in a same fashion. Recall the tangent space of $\tensM_1$ in~\cref{eq: tangent space of gen St}. The orthogonal complement with regard to the metric~\cref{eq: CCA metric new} of the tangent space $\tangent_{\matU}\!\tensM_1$ can be characterized by
\begin{equation}
  \label{eq: normal space to new metric}
  (\tangent_{\matU}\!\tensM_1)^\perp=\{\matU\matS_1^{}\matM_{1,2}^{-1}:\matS_1\in\mathbb{R}^{m\times m},\ \matS_1^{}=\matS_1^\T\},
\end{equation}
since the dimension of $\{\matU\matS_1^{}\matM_{1,2}^{-1}:\matS_1\in\mathbb{R}^{m\times m},\ \matS_1=\matS_1^\T\}$ is $m(m+1)/2$ and $\tr((\matU\matS_1^{}\matM_{1,2}^{-1})^\T\Sigma_{xx}(\matu\matOmega_1+\matu_{\Sigma_{xx}\perp}\matK_1)\matM_{1,2})=0$ holds for all $\matS_1,\matOmega_1,\matK_1$ satisfying that $\matS_1=\matS_1^\T$ and $\matOmega_1=-\matOmega_1^\T$. Moreover, in the light of $\tangent_{\matU}\!\tensM_1\oplus(\tangent_{\matU}\!\tensM_1)^\perp=\tangent_{\matU}\!\mathbb{R}^{d_x\times m}\simeq\mathbb{R}^{d_x\times m}$, there is a unique orthogonal decomposition for $\bar{\eta}_1\in\mathbb{R}^{d_x\times m}$
\begin{equation}
  \bareta_1=\Pi_{\mathrm{new},\matu}^{}(\bareta_1)+\Pi_{\mathrm{new},\matu}^\perp(\bareta_1)=(\matu\matOmega_1+\matu_{\Sigma_{xx}\perp}\matK_1)+\matU\matS_1^{}\matM_{1,2}^{-1},\label{eq: CCA orth decomp}
\end{equation}
i.e., $\Pi_{\mathrm{new},\matu}^{}(\bareta_1)=\bareta_1-\Pi_{\mathrm{new},\matu}^\perp(\bareta_1)=\bareta_1-\matu\matS_1^{}\matM_{1,2}^{-1}$. To characterize the symmetric matrix $\matS_1$, we multiply~\cref{eq: CCA orth decomp} from the left by $\matu^\T\Sigma_{xx}$, and yield $\matu^\T\Sigma_{xx}\bareta_1=\matOmega_1+\matS_1^{}\matM_{1,2}^{-1}$. Summing up $\matu^\T\Sigma_{xx}\bareta_1$ and $(\matu^\T\Sigma_{xx}\bareta_1)^\T$, we obtain that $\matS_1^{}\matM_{1,2}^{-1}+\matM_{1,2}^{-1}\matS_1^{}=\matu^\T\Sigma_{xx}\bareta_1+\bareta_1^\T\Sigma_{xx}\matu$, which has a unique solution according to~\cite[Theorem 2.4.4.1]{horn2012matrix}.
\end{proof}

\section{Implementation details of TR-GN}\label{app: TR-GN}
Recall that the search direction $\eta^{(t)}$ in~\cref{alg: RGN} is determined by the following least-squares problem
\begin{equation}
  \underset{\eta\in\tangent_{\vecmatW}\!\tensM}{\arg\min}\,
  \|\mathrm{D}F(\vecmatW)[\eta]+F(\vecmatW)\|_\mathrm{F}^2.\label{eq: search direction by ls}
\end{equation}
Specifically, it follows from the multilinearity of $\tau$ that the directional derivative $\mathrm{D}F(\vecmatW)[\eta]$ in~\cref{eq: search direction by ls} can be computed by
\begin{equation*}
  \begin{aligned}
    \mathrm{D}F(\vecmatW)[\eta]&=\lim_{h\to 0}\frac{\proj_{\Omega} (\tau(\vecmatW+h\eta)-\tensA)-\proj_{\Omega} (\tau(\vecmatW)-\tensA)}{\sqrt{p}h}\\
    &=\frac{1}{\sqrt{p}}\sum_{k=1}^{d}\proj_{\Omega} (\tau(\matW_1,\dots,\matW_{k-1},\eta_k,\matW_{k+1},\dots,\matW_d)).
  \end{aligned}
\end{equation*}
Then, we yield
\begin{equation*}
  \begin{aligned}
    &\|\mathrm{D}F(\vecmatW)[\eta]+F(\vecmatW)\|_\mathrm{F}^2\\
    &=\frac{1}{p}\sum_{i=1}^{n_1n_2\cdots n_d}\langle\proj_\Omega(\tensB_i),\sum_{k=1}^{d}\tau(\matW_1,\dots,\matW_{k-1},\eta_k,\matW_{k+1},\dots,\matW_d)+\tau(\vecmatW)-\tensA\rangle^2\\
    &=\frac{1}{p}\sum_{i=1}^{n_1n_2\cdots n_d}\Big(\sum_{k=1}^d\left\langle\proj_{\Omega_{(k)}}\!((\tensB_i)_{(k)})\matW_{\neq k},\eta_k\right\rangle+\left\langle\proj_\Omega(\tensB_i),\tau(\vecmatW)-\tensA\right\rangle\Big)^2\\
    &=\frac{1}{p}\sum_{(i_1,\dots,i_d)\in\Omega}\Big(\sum_{k=1}^d\eta_k(i_k,:)^\T\rmvec((\prod_{j=k+1}^{d}\matU_{j}(i_{j})\prod_{j=1}^{k-1}\matU_{j}(i_{j}))^\T)+\tensS(i_1,\dots,i_d)\Big)^2,
  \end{aligned}
\end{equation*}
where $\{\tensB_i\}_{i=1}^{n_1n_2\cdots n_d}$ is defined by $(\tensB_i)(i_1,i_2,\dots,i_d)=1$ if $i=\sum_{j=1}^d (i_j-1)\prod_{\ell=1}^{j-1}n_\ell$, otherwise $(\tensB_i)(i_1,i_2,\dots,i_d)=0$, and $\tensS:=\proj_\Omega(\tau(\vecmatW)-\tensA)$ refers to the residual tensor. Note that for $i=\sum_{j=1}^d (i_j-1)\prod_{\ell=1}^{j-1}n_\ell$, the matrix $\proj_{\Omega_{(k)}}\!((\tensB_i)_{(k)})=0$ if $(i_1,i_2,\dots,i_d)\notin\Omega$. Consequently, the problem~\cref{eq: search direction by ls} is a least-squares problem of $\sum_{k=1}^d n_kr_{k-1}r_k$ variables.

\bibliographystyle{siamplain}
\bibliography{references}
\end{document}